\documentclass[a4paper,10pt]{article}
\usepackage{amsfonts}
\usepackage{amsmath,amsthm}
\usepackage{amssymb}
\usepackage{thmdefs}
\usepackage[round,authoryear]{natbib}
\usepackage{epsfig}
\usepackage{fancyhdr}
\usepackage{verbatim}
\usepackage[latin1]{inputenc}

\numberwithin{equation}{section}
\numberwithin{equation}{section} 
\begin{document}

\title{Discrete Dispersion Models and their Tweedie Asymptotics}
\author{Bent Jørgensen \\
University of Southern Denmark, \\
Department of Mathematics and Computer Science, \\
Campusvej 55, DK-5230 Odense M, Denmark; \\
E-mail : bentj@stat.sdu.dk \and Célestin C. Kokonendji \\
Université de Franche-Comté - UFR Sciences et Techniques \\
Laboratoire de Mathématiques de Besançon - UMR 6623 CNRS \\
16, route de Gray - 25030 Besançon cedex, France; \\
E-mail : celestin.kokonendji@univ-fcomte.fr}
\maketitle

\begin{abstract}
We introduce a class of two-parameter discrete dispersion models, obtained
by combining convolution with a factorial tilting operation, similar to
exponential dispersion models which combine convolution and exponential
tilting. The equidispersed Poisson model has a special place in this
approach, whereas several overdispersed discrete distributions, such as the
Neyman Type A, Pólya-Aeppli, negative binomial and Poisson-inverse Gaussian,
turn out to be Poisson-Tweedie factorial dispersion models with power
dispersion functions, analogous to ordinary Tweedie exponential dispersion
models with power variance functions. Using the factorial cumulant
generating function as tool, we introduce a dilation operation as a discrete
analogue of scaling, generalizing binomial thinning. The Poisson-Tweedie
factorial dispersion models are closed under dilation, which in turn leads
to a Poisson-Tweedie asymptotic framework where Poisson-Tweedie models
appear as dilation limits. This unifies many discrete convergence results
and leads to Poisson and Hermite convergence results, similar to the law of
large numbers and the central limit theorem, respectively. The dilation
operator also leads to a duality transformation which in some cases
transforms overdispersion into underdispersion and vice-versa. Many of the
results have multivariate analogues, and in particular we consider a class
of multivariate Poisson-Tweedie models, a multivariate notion of over- and
underdispersion, and a multivariate zero-inflation index.\medskip\ \newline
\end{abstract}

\section{Introduction}

Given the plethora of discrete distributions available in the literature 
\nocite{Johnson2005}
\nocite{Wimmer1999}(Johnson \emph{et al.}, 2005;\ Wimmer and Altmann, 1999),
it is difficult to point, with conviction, to one or the other two-parameter
discrete family as being especially suited for modelling count data
phenomena such as over/underdispersion or zero-inflation/deflation. The
central limit theorem leads to the normal distribution, which is continuous,
and there are few general discrete asymptotic results available other than
conventional Poisson convergence. Echoing \cite{Tweedie1984}, who introduced
the family of continuous Tweedie models now bearing his name, we should
perhaps be looking for \emph{an index which distinguishes between some
important discrete distributions}. Ideally, such a class of discrete
distributions should be justified by a general asymptotic result like the
Tweedie convergence theorem of \cite{Jorgensen1994}.

There are several problems that make the discrete case more difficult to
handle than the continuous case. The first problem is that there are no
immediate discrete analogues of location and scale transformations, which
are crucial in the continuous case for handling scaling limits such as the
central limit theorem. A second and related problem is that there are no
obvious discrete analogues of standard continuous distributions such as the
normal or gamma distributions. A third problem is that discrete natural
exponential families (power-series distributions), while ubiquitous, tend to
have much more complicated variance functions than in the continuous case.

An important step forward was taken by \cite{Steutel1979}, who introduced
the discrete analogue of positive stable distributions by using binomial
thinning instead of scaling. The same technique has been used extensively
for constructing discrete time-series models 
\citep[e.g.][]{Weiss2008}%
. Recently, \cite{Harremoes2010} used binomial thinning to formulate an
extended Poisson convergence theorem, which they called the "law of thin
numbers", whereas \cite{Puig2003} and \cite{Puig2006,Puig2007} have
characterized discrete distributions closed under convolution and binomial
thinning.

In order to make further progress, we shall follow the footsteps of \cite%
{Jorgensen2007a} and \cite{Jorgensen2010}, who developed analogues of
Tweedie asymptotics for extremes and geometric sums, respectively. These
authors explored specialized versions of the cumulant generating function
(CGF), and showed that each of the two corresponding analogues of the
variance function are efficient characterization and convergence tools.

In the present paper we argue that the factorial cumulant generating
function (FCGF) is the most suitable choice for handling the discrete case,
along with the first two factorial cumulants, namely the mean and the
dispersion. Firstly, the FCGF characterizes convolution additively.
Secondly, we shall use the FCGF to generalize binomial thinning to a
dilation operator, providing the discrete analogue of scaling. Thirdly, the 
\emph{dispersion function}, which expresses the dispersion as a function of
the mean, leads to a new discrete Poisson-Tweedie convergence theorem. Many
known discrete distributions such as the Hermite, Neyman Type A, Pó%
lya-Aeppli, binomial, negative binomial and Poisson-Inverse Gaussian
distributions have power dispersion functions, and hence appear as limits in
the corresponding regime of power asymptotics for dispersion functions. The
corresponding power parameter is the index alluded to above.

The plan of the paper is to develop a new class of factorial dispersion
models and Poisson-Tweedie mixtures as analogues of conventional exponential
and Tweedie dispersion models, respectively, along the lines of 
\citet[Ch.\ 3-4]{Jorgensen1997}%
. We review FCGFs and factorial cumulants in Section~\ref{sec:cumulant}, we
consider Poisson and Hermite convergence, and we consider the concept of
infinite dilatability and its relation with Poisson mixtures. We introduce a
new operation called the M-transformation, and show that in some cases it
presents a duality between over- and underdispersion. In Section~\ref{sec3}
we consider a new factorial tilting operation and introduce the class of
factorial dispersion models and their dispersion functions. We show that the
Poisson-Tweedie mixtures are factorial dispersion models and show that their
dispersion functions are of power form. In Section~\ref{sec5} we present a
general convergence theorem for dispersion functions (with proof given in
Appendix B) and present the new Poisson-Tweedie convergence theorem and some
examples. We consider the multivariate case in Section~\ref{sec:Multivariate}%
, where we discuss multivariate factorial cumulants and some of their
properties, and consider multivariate over-, equi-, and underdispersion. We
also introduce a new class of multivariate Poisson-Tweedie mixtures, which
provides multivariate versions of many of the distributions mentioned above.
Finally, Appendix A contains a summary of relevant results for exponential
dispersion models.

\section{Factorial cumulant generating functions\label{sec:cumulant}}

We begin by developing basic results for the FCGF and factorial cumulants,
and use them to prove the law of thin numbers and Hermite convergence, which
are discrete analogues of the law of large numbers and the central limit
theorem, respectively. We define Poisson translation and dilation, and
discuss infinite dilatability and its relation with Poisson mixtures. We
also introduce the M-transformation and discuss its relation with
over/underdispersion. Many results in the following deal with the \emph{%
discrete case}, meaning non-negative integer-valued random variables, but
unless otherwise indicated, results are valid for general random variables.

\subsection{Cumulant generating functions\label{sec:general}}

The ordinary cumulant generating function (CGF) for a random variable $X$ is
defined by 
\begin{equation*}
\kappa (s)=\kappa (s;X)=\log \mathrm{E}(e^{sX})\text{ for }s\in \mathbb{R}%
\text{,}
\end{equation*}%
with effective domain $\mathrm{dom}(\kappa )=\left\{ s\in \mathbb{R}:\kappa
(s)<\infty \right\} $. The CGF satisfies the linear transformation law 
\begin{equation}
\kappa (t;aX+b)=\kappa (at;X)+bt,  \label{linear}
\end{equation}%
which is crucial for asymptotic results like the law of large numbers and
the central limit theorem.

To obtain a discrete analogue of (\ref{linear}), we consider the \emph{%
factorial cumulant generating function} (FCGF)\emph{\ }for $X$ 
\citep[p.~55]{Johnson2005}%
, defined by 
\begin{equation}
C(t)=C(t;X)=\log \mathrm{E}\left[ \left( 1+t\right) ^{X}\right] =\kappa
(\log \left( 1+t\right) ;X)\text{ for }t>-1\text{,}  \label{see}
\end{equation}%
with effective domain $\mathrm{dom}(C)=\left\{ t>-1:C(t)<\infty \right\}
=\exp \left[ \mathrm{dom}(\kappa )\right] -1$. We also note that $C$, like $%
\kappa $, characterizes convolution additively, i.e. for independent random
variables $X$ and $Y$ we have 
\begin{equation}
C(t;X+Y)=C(t;X)+C(t;Y)\text{.}  \label{plus}
\end{equation}

The CGF $\kappa $ is a real analytic convex function, and strictly convex
unless $X$ is degenerate. Hence, $C$ is also real analytic, and the domain $%
\mathrm{dom}(C)$, like $\mathrm{dom}(\kappa )$, is an interval. The
derivative $\dot{C}(t)=\dot{\kappa}(\log \left( 1+t\right) )/\left(
1+t\right) $ has the same sign as $\dot{\kappa}(\log \left( 1+t\right) )$ on 
$\mathrm{int}\left( \mathrm{dom}(C)\right) $. Hence, by the convexity of $%
\kappa $, the FCGF $C$ is either monotone or u-shaped. Let $\mathcal{K}$
denote the set of CGFs $\kappa $ such that $\mathrm{int}(\mathrm{dom}(\kappa
))\neq \emptyset $, and let $\mathcal{C}$ denote the corresponding set of
FCGFs $C$ of the form (\ref{see}) with $\mathrm{int}(\mathrm{dom}(C))\neq
\emptyset $. In this case, either of the functions $\kappa $ or $C$
characterizes the distribution of $X$, and in equations like (\ref{see}), we
assume that equality holds in a neighbourhood of zero. From now on, CGF and
FCGF refer to functions in $\mathcal{K}$ and $\mathcal{C}$, respectively.

\subsection{Dilation and Poisson translation\label{sec:Dil}}

In order to obtain a discrete analogue of scaling, we define the \emph{%
dilation} $c\cdot X$ of a random variable $X$ by 
\begin{equation}
C(t;c\cdot X)=C(ct;X)\text{,}  \label{GCF}
\end{equation}%
for scalars $c>0$ such that right-hand side of (\ref{GCF}) is an FCGF. We
say that $X$ (or its distribution) is \emph{infinitely dilatable} if the
dilation $c\cdot X$ exists for any $c>0$. For $X$ discrete and $0<c<1,$ the
dilation corresponds to \emph{binomial thinning}, 
\begin{equation}
c\cdot X\overset{D}{=}\sum\limits_{i=1}^{X}N_{i}\text{,}  \label{geosum}
\end{equation}%
where $\overset{D}{=}$ denotes equality in distribution, and where $%
N_{1},N_{2},\ldots $ is a sequence of i.i.d.~Bernoulli random variables with
probability parameter $c$, independent of $X$. To prove (\ref{geosum}), we
note that the Bernoulli FCGF is%
\begin{equation}
C(t;N_{1})=\log \left( 1+ct\right) \text{,}  \label{geoc}
\end{equation}%
whereby 
\begin{equation*}
\log \mathrm{E}\left[ \left( 1+t\right) ^{c\cdot X}\right] =\log \mathrm{E}%
\left[ \left( 1+ct\right) ^{X}\right] =C(ct;X)\text{,}
\end{equation*}%
which implies (\ref{GCF}). We note in passing, that the Bernoulli FCGF (\ref%
{geoc}) is not infinitely dilatable, due to the constraint $c<1$, whereas
the geometric distribution with FCGF 
\begin{equation}
C(t)=-\log \left( 1-\mu t\right)  \label{geometric}
\end{equation}%
is infinitely dilatable, since the domain for $\mu $ in (\ref{geometric}) is 
$\mathbb{R}_{+}$. The binomial thinning operator is an important tool for
constructing discrete time series models, see e.g. Weiß\ (2008)\nocite%
{Weiss2008} and references therein.

A further extension of the dilation operator may be obtained by means of
geometric compounding. Let us assume that, conditionally on a non-negative
random variable $X$, we have a negative binomial FCGF $-X\log \left(
1-ct\right) $, where $c>0$. The resulting negative binomial compound
variable $Y$ has FCGF 
\begin{equation}
C(t;Y)=\log \mathrm{E}\left[ \left( 1-\mu t\right) ^{-X}\right] =C(-ct;-X),
\label{sign}
\end{equation}%
which is clearly infinitely dilatable. In the case where $X$ is discrete,
this is a geometric compounding of the form $N_{1}+\cdots +N_{X}$, where $%
N_{1},N_{2},\ldots $ is a sequence of i.i.d.~geometric random variables with
mean $c,$ independent of $X$, corresponding to the negative binomial
thinning of \cite{Ristic2009} and \cite{Barreto-Souza2014}. For $c=1$ and
reversing the sign of $t$ in (\ref{sign}) we obtain $C(t;-X)=C(-t;Y),$
providing a possible interpretation of the reflection operator.

The dilation operator satisfies the following associative property: 
\begin{equation*}
c_{2}\cdot \left( c_{1}\cdot X\right) \overset{D}{=}\left( c_{1}c_{2}\right)
\cdot X\text{,}
\end{equation*}%
provided that the left-hand side of the equation exists. We also note the
following distributive property of dilation for independent random variables 
$X$ and $Y$, 
\begin{equation*}
c\cdot \left( X+Y\right) \overset{D}{=}c\cdot X+c\cdot Y\text{.}
\end{equation*}

To obtain a discrete analogue of translation, we note that the Poisson
distribution $\mathrm{Po}(\mu )$ with mean $\mu \geq 0$ has FCGF 
\begin{equation}
C(t;\mathrm{Po}(\mu ))=\mu t\text{ for }t>-1\text{,}  \label{exponential}
\end{equation}%
including the degenerate case $\mathrm{Po}(0)\equiv 0$. The Poisson FCGF is
hence analogous to the CGF of a constant, and it is infinitely dilatable. We
define the \emph{Poisson translation} operator $\oplus \mu $ for $\mu \geq 0$
by convolution, i.e. 
\begin{equation*}
C(t;X\oplus \mu )=C(t;X)+\mu t\text{.}
\end{equation*}%
As an example, we may write the stationary Poisson INAR(1) time series model 
\citep[cf.][]{McKenzie1985}
in the following way: 
\begin{equation*}
X_{t}=c\cdot X_{t-1}\oplus \left[ \lambda \left( 1-c\right) \right]
\end{equation*}%
where $X_{t}\sim \mathrm{Po}(\lambda )$ for $t=0,1,\ldots $, and $0<c<1$. We
also define the \emph{Poisson subtraction} $\ominus \mu $ by 
\begin{equation}
C(t;X\ominus \mu )=C(t;X)-\mu t  \label{substral}
\end{equation}%
for values $\mu \geq 0$ such that the right-hand side of (\ref{substral}) is
an FCGF. As an example, we consider the \emph{Short} distribution \nocite%
{Johnson2005}(Johnson \emph{et al.}, 2005, p. ~419) with FCGF 
\begin{equation*}
C(t;X)=\mu _{1}\left( e^{\phi t}-1\right) +\mu _{2}t
\end{equation*}%
with $\mu _{1},\mu _{2},\phi >0$. In this case, the Poisson subtraction (\ref%
{substral}) exists for $\mu \leq \mu _{2}$.

\subsection{Factorial cumulants}

When $0\in \mathrm{int}(\mathrm{dom}(C))$, the derivatives $%
C^{(n)}(0)=C^{(n)}(0;X)$ are the \emph{factorial cumulants} of $X$, which
have many analogies with ordinary cumulants. The first factorial cumulant is
the mean $\mathrm{E}(X)=\dot{C}(0;X)$. The second factorial cumulant $%
\mathrm{S}(X)$, defined by%
\begin{equation*}
\mathrm{S}(X)=\ddot{C}(0;X)=\ddot{\kappa}(0;X)-\dot{\kappa}(0;X)=\mathrm{Var}%
\left( X\right) -\mathrm{E}\left( X\right) ,
\end{equation*}%
is denoted the \emph{dispersion} for $X$ (compare with \emph{Fisher's\
dispersion index} $\mathrm{D}\left( X\right) =\mathrm{Var}\left( X\right) /%
\mathrm{E}\left( X\right) $). The dispersion is bounded below by the
negative expectation, 
\begin{equation}
-\mathrm{E}(X)\leq \mathrm{S}(X).  \label{lessthan}
\end{equation}%
The dispersion $\mathrm{S}\left( X\right) $ indicates \emph{underdispersion}
if $-\mathrm{E}\left( X\right) \leq \mathrm{S}\left( X\right) <0$, \emph{%
equidispersion} if $\mathrm{S}\left( X\right) =0$, and \emph{overdispersion}
if $\mathrm{S}\left( X\right) >0,$ relative to the Poisson distribution.

The mean and dispersion satisfy the following transformation laws: 
\begin{equation}
\mathrm{E}\left( c\cdot X\oplus \mu \right) =c\mathrm{E}(X)+\mu \text{,}%
\qquad \mathrm{S}\left( c\cdot X\oplus \mu \right) =c^{2}\mathrm{S}(X),
\label{factor}
\end{equation}%
similar to the mean and variance of the linear transformation (\ref{linear}%
). Furthermore the $n$th factorial cumulant is homogeneous of degree $n$
with respect to dilation, i.e. $C^{(n)}(0;c\cdot X)=c^{n}C^{(n)}(0;X)$. For
general $X$ and $Y$ we obtain 
\begin{equation*}
\mathrm{S}(X+Y)=\mathrm{S}(X)+\mathrm{S}(Y)+2\mathrm{Cov}(X,Y)\text{,}
\end{equation*}%
which follows from the corresponding formula for the variance. In
particular, the dispersion is additive for uncorrelated random variables.

Applying the transformation laws (\ref{factor}) to the dilated variable $%
c\cdot X$, the inequality (\ref{lessthan}) implies that 
\begin{equation*}
-\mathrm{E}(X)\leq c\mathrm{S}(X)\text{.}
\end{equation*}%
Hence if $\mathrm{E}(X)$ and $\mathrm{S}(X)$ have opposite signs, the value
of $c$ is bounded from above or below, in which case $X$ cannot be
infinitely dilatable. In particular, if $\mathrm{E}\left( X\right) >0$ and $%
X $ is underdispersed, then $X$ is not infinitely dilatable, an example
being the Bernoulli distribution (\ref{geoc}).

Another index that may be obtained from the FCGF is the \emph{zero-inflation
index}, defined for a discrete random variable $X$ by 
\begin{equation}
\mathrm{ZI}\left( X\right) =1+\frac{\log P(X=0)}{\mathrm{E}\left( X\right) }%
=1+\frac{C(-1;X)}{\dot{C}(0;X)},  \label{ZI}
\end{equation}%
cf. Puig \& Valero 2006, 2007)\nocite{Puig2006,Puig2007}. The index $\mathrm{%
ZI}\left( X\right) $ indicates zero-inflation if $\mathrm{ZI}\left( X\right)
>0$ and zero-deflation if $\mathrm{ZI}\left( X\right) <0$, relative to the
Poisson distribution for which $\mathrm{ZI}\left( X\right) =0$. We consider
a multivariate generalization of $\mathrm{ZI}\left( X\right) $ in Section %
\ref{sec:fcgf}.

\subsection{Poisson and Hermite convergence\label{sec:Poisson}}

We shall now present discrete analogues of the law of large numbers and the
central limit theorem, obtained by exploring dilation and Poisson
translation/subtraction. First we present a new proof of the \emph{Law of
Thin Numbers} due to \cite{Harremoes2010}, which is a kind of Poisson law of
large numbers for discrete distributions, in the sense that the Poisson
distribution plays the role of a degenerate distribution. Our proof is based
on the FCGF, whereas \cite{Harremoes2010} used direct methods in their proof.

We define the \emph{dilation average} for i.i.d. sequence $%
X_{1},X_{2},\ldots $ by 
\begin{equation}
\overline{X}_{n}=n^{-1}\cdot \left( X_{1}+\cdots +X_{n}\right) \text{.}
\label{dilatione}
\end{equation}%
For discrete random variables the dilation in (\ref{dilatione}) is defined
by binomial thinning, because $n^{-1}\leq 1$.

\begin{proposition}[Law of Thin Numbers]
\label{LTN}Let $X_{1},X_{2},\ldots $ denote an i.i.d. sequence of discrete
random variables with mean $\mu >0$. Then the dilation average $\overline{X}%
_{n}$ converges in distribution to a Poisson distribution: 
\begin{equation}
\overline{X}_{n}\overset{D}{\rightarrow }\mathrm{Po}(\mu )\text{ as }%
n\rightarrow \infty \text{.}  \label{poisslim}
\end{equation}
\end{proposition}

\begin{proof}
By using the additive property (\ref{plus}) along with the
definition of dilation, we obtain 
\begin{eqnarray}
C\left( t;\overline{X}_{n}\right) &=&C\left( n^{-1}t;X_{1}+\cdots
+X_{n}\right)  \notag \\
&=&nC(n^{-1}t;X_{1})=\mu t+O(n^{-1}),  \label{cconv}
\end{eqnarray}%
which converges to the Poisson FCGF as $n\rightarrow \infty $. Since $0\in 
\mathrm{int}\left( \mathrm{dom}C\left( \cdot ;\overline{X}_{n}\right)
\right) $, we conclude from Theorem 1 of \cite{Jensen1997} that there exists
a probability measure $P$ such that the sequence of probability measures $%
P_{n}$ corresponding to $C\left( \cdot ;\overline{X}_{n}\right) $ converges
weakly to $P$. It follows that the sequence $\exp C\left( e^{s}-1;\overline{X%
}_{n}\right) $ converges to the moment generating function (MGF) of $P$ for $%
e^{s}-1\in \mathrm{dom}C\left( \cdot ;\overline{X}_{n}\right) $, which in
view of (\ref{cconv}) implies that $P$ is the Poisson distribution $\mathrm{%
Po}(\mu ),$ completing the proof.
\end{proof}

For a constant integer variable $n$, the thinned variable $c\cdot n$ is
binomial $\mathrm{Bi}(c,n)$, corresponding to the special case $X_{i}\equiv
1 $ of (\ref{dilatione}). We hence obtain the classical Poisson convergence
theorem as a corollary, albeit in a rather terse notation.

\begin{corollary}[Poisson Convergence]
For $\mu >0$ we obtain 
\begin{equation*}
\left( \frac{\mu }{n}\right) \cdot n\overset{D}{\rightarrow }\mathrm{Po}(\mu
)\text{ as }n\rightarrow \infty \text{.}
\end{equation*}
\end{corollary}

Before turning to Hermite convergence, we consider the Hermite distribution.

\begin{example}[Hermite distribution]
\label{Lapl}The Hermite distribution 
\citep{Kemp1965}%
, denoted \textrm{$\mathrm{PT}$}$_{0}(\mu ,\gamma )$ (conforming with the
notation of Section \ref{sec3.6}) is defined by the FCGF 
\begin{equation}
C(t)=\frac{\gamma }{2}t^{2}+\mu t,  \label{geotrans}
\end{equation}%
where $\mu >0$ is the mean, and $\gamma $ is the dispersion, satisfying $%
0<\gamma \leq \mu $. This restriction on the parameters follows from the
corresponding log PGF 
\begin{equation*}
C(u-1)=\frac{\gamma }{2}\left( u^{2}-1\right) +\left( \mu -\gamma \right)
\left( u-1\right) \text{,}
\end{equation*}%
whose coefficients $\gamma $ and $\mu -\gamma $ must both be non-negative 
\citep{Kemp1965}%
. The Hermite distribution is a discrete analogue of the normal
distribution, in the sense that its third and higher order factorial
cumulants are all zero. It is, however, rather different in nature from
other discrete normal distributions such as those proposed by \cite{Kemp1997}
and Roy (2003)\nocite{Roy2003}. See Giles (2010)\nocite{Giles2010} and Puig
\& Barquinero (2011)\nocite{Puig2011} for applications of the Hermite
distribution.
\end{example}

In order to obtain an analogue of the central limit theorem, we propose to
use Poisson translation and dilation instead of centering and scaling. We
consider the partial sum $S_{n}=X_{1}+\cdots +X_{n}$ based on i.i.d.
discrete variables $X_{i}$ with $\mathrm{E}(X_{i})=m>0$ and $\mathrm{S}%
(X_{i})=\gamma >0$. A formal analogy of the conventional centered and scaled
partial sum of the form 
\begin{equation*}
Z_{n}=n^{-1/2}\cdot \left( S_{n}\ominus nm\right)
\end{equation*}%
is, however, degenerate, because the centered sum $S_{n}\ominus nm$, is a
non-negative variable with mean zero. By adding a constant $\mu \geq \gamma $%
, we obtain, formally,%
\begin{eqnarray}
Z_{n}\oplus \mu &=&n^{-1/2}\cdot \left( S_{n}\ominus nm\right) \oplus \mu 
\notag \\
&=&n^{-1/2}\cdot \left[ S_{n}\ominus \left( nm-n^{1/2}\mu \right) \right] 
\text{.}  \label{cltmu}
\end{eqnarray}%
The constant being subtracted in (\ref{cltmu}) is now $nm-n^{-1/2}\mu <nm$,
thereby avoiding the above degeneracy, at the cost of a constant Poisson
translation. The expression (\ref{cltmu}) has the further advantage that the
dilation by $n^{-1/2}\leq 1$ is again defined by binomial thinning. This
leads us to the following analogue of the central limit theorem.

\begin{proposition}[Hermite Convergence Theorem]
\label{clt}Let $S_{n}=X_{1}+\cdots +X_{n}$ denote the partial sum for an
i.i.d. sequence of discrete random variables $X_{i}$ with $\mathrm{E}%
(X_{1})=m>0$ and $\mathrm{S}(X_{1})=\gamma >0$. Define the translated
standardized variable $Z_{n}\left( \mu \right) $ for $\mu \geq \gamma $ by%
\begin{equation*}
Z_{n}\left( \mu \right) =n^{-1/2}\cdot \left[ S_{n}\ominus \left(
nm-n^{1/2}\mu \right) \right]
\end{equation*}%
Then $Z_{n}\left( \mu \right) $ converges in distribution to the Hermite
distribution $\mathrm{PT}_{0}(\mu ,\gamma )$ as $n\rightarrow \infty $.
\end{proposition}

\begin{proof}
The proof follows by expanding the FCGF of $Z_{n}\left( \mu \right) $ as
follows: 
\begin{equation*}
C(t;Z_{n}\left( \mu \right) )=\mu t+\frac{\gamma }{2}t^{2}+O\left(
n^{-1/2}\right) \text{,}
\end{equation*}%
which shows that the Hermite FCGF (\ref{geotrans}) with $\mu \geq \gamma $
appears in the limit as $n\rightarrow \infty $. Using once more the results
of \cite{Jensen1997}, we conclude that $Z_{n}\left( \mu \right) $ converges
in distribution to the Hermite distribution $\mathrm{PT}_{0}(\mu ,\gamma )$. 
\end{proof}

A number of further convergence results will be considered in Section \ref%
{sec5}.

\subsection{Poisson mixtures and infinite dilatability\label{sec:Mixtures}}

We now discuss the relation between infinite dilatability and Poisson
mixtures. If $X$ is a non-negative random variable, and $a\geq 0$, we define
the Poisson mixture $P\left( X;a\right) $ by the following conditional
distribution: 
\begin{equation*}
P\left( X;a\right) |X\sim \mathrm{Po}(aX),
\end{equation*}%
see also \cite{Karlis2005}. The corresponding conditional moment generating
function (MGF) is 
\begin{equation*}
\mathrm{E}\left( e^{sP\left( X;a\right) }|X\right) =\exp \left( aX\left(
e^{s}-1\right) \right) .
\end{equation*}%
Hence $P\left( X;a\right) $ has MGF 
\begin{equation*}
\mathrm{E}\left( e^{sP\left( X;a\right) }\right) =\mathrm{E}\left[ \exp
\left( aX\left( e^{s}-1\right) \right) \right] =e^{\kappa (a\left(
e^{s}-1\right) ;X)}\text{,}
\end{equation*}%
which implies that 
\begin{equation}
C(t;P\left( X;a\right) )=\kappa (ta;X)\text{.}  \label{mixexpo}
\end{equation}%
It follows that the factorial cumulants of $P\left( X;a\right) $ are
obtained by scaling the ordinary cumulants for $X$. In particular, the first
and second factorial cumulants are%
\begin{equation*}
\mathrm{E}\left( P\left( X;a\right) \right) =a\mathrm{E}\left( X\right) 
\text{ and }\mathrm{S}\left( P\left( X;a\right) \right) =a^{2}\mathrm{Var}%
\left( X\right) \text{,}
\end{equation*}%
making $P\left( X;a\right) $ overdispersed, unless $aX$ is degenerate.

In view of the scaling property $\kappa (ta;X)=\kappa (t;aX)$ of the CGF it
follows from (\ref{mixexpo}) that any Poisson mixture $P\left( X;a\right) $
is infinitely dilatable. The following result also contains the converse
implication, similar to Theorem 3.1 of \cite{Jorgensen2010} for geometric
infinite divisibility. This is an important prerequisite for our discussion
of factorial tilting families in Section \ref{sec3}. We recall the
definitions of the space $\mathcal{K}$ of CGFs and the space $\mathcal{C}$
of FCGFs, cf. Section \ref{sec:general}.

\begin{theorem}
\label{thinfdiv}Let the FCGF $C\in \mathcal{C}$ be given. Then the following
conditions are equivalent:

\begin{enumerate}
\item $C$ is an infinitely dilatable FCGF;

\item $C\left( c\,\cdot \right) \in \mathcal{C}$ for any $c>0$;

\item $C\left( c\,\cdot \right) \in \mathcal{K}$ for any $c>0$;

\item $C$ is the FCGF for a Poisson mixture $P\left( X;a\right) $.
\end{enumerate}
\end{theorem}

\begin{proof}
1. $\Leftrightarrow $ 2.: This is the definition of infinite dilatability.

2. $\Rightarrow $ 3.: Condition 2. implies that $C(n\,\cdot )\in \mathcal{C}$
for any integer $n$, so in view of (\ref{see}) we find that $C(n\left(
e^{sc/n}-1\right) )$ is a CGF for any $c>0$. Letting $n\rightarrow \infty $
we obtain in the limit the function $C\left( c\,\cdot \right) $, which is
hence a CGF for any $c>0,$ implying 3.

3. $\Rightarrow $ 4.: By condition 3. we have that $C$ is a CGF. The
corresponding Poisson mixture (\ref{mixexpo}) has FCGF $C\left( \cdot
\right) $, which implies 4.

4. $\Rightarrow $ 1.: This implication follows from (\ref{mixexpo}) because $%
C(ct;P\left( X;a\right) )=C(t;P\left( X;ca\right) )$ for any $c>0$, which in
turn implies that the Poisson mixture $P\left( X;a\right) $ is infinitely
dilatable. This completes the proof.
\end{proof}

\begin{corollary}
Any infinitely dilatable FCGF $C\in \mathcal{C}$ is convex, and strictly
convex except in the Poisson case.
\end{corollary}

\begin{proof}
Since $C\in \mathcal{K}$ in the infinitely dilatable case, it follows that $%
C $ is convex, and strictly convex unless $C(t)=\mu t$ for some $\mu \geq 0$%
, corresponding to the Poisson case (\ref{exponential}).
\end{proof}

For example, the geometric FCGF (\ref{geometric}), being a Poisson mixture
and hence infinitely dilatable, is strictly convex. Conversely, the
Bernoulli FCGF (\ref{geoc}), being strictly concave, is not infinitely
dilatable, as we already know.

A Poisson mixture $P\left( X;a\right) $ may be expressed as a weighted
Poisson distribution (\cite{Kokonendji2008,Kokonendji2012}) 
\begin{equation}
f_{w}(x;\mu )=\frac{w(x;\mu )\mu ^{x}e^{-\mu }}{\mathrm{E}_{\mu }\left[ w(X)%
\right] x!}\text{ for }x=0,1,\ldots  \label{prob}
\end{equation}%
where the weights have the form 
\begin{equation*}
w(x;\mu )=e^{\mu }\left( -1\right) ^{x}\frac{d^{x}}{d\mu ^{x}}\mathrm{E}%
\left( e^{-\mu X}\right) \text{.}
\end{equation*}%
The probabilities (\ref{prob}) may hence be calculated from the MGF for $X$.

\subsection{The M-transformation and over/underdispersion\label{sec:M-trans}}

We now introduce a transformation that in some cases transforms
underdispersion into overdispersion and vice-versa. Let us consider the
reflected variable$-X$ with FCGF%
\begin{equation*}
C(t;-X)=\log \mathrm{E}\left[ \left( 1+t\right) ^{-X}\right] =C\left( \frac{%
-t}{1+t};X\right) \text{ for }t>-1.
\end{equation*}%
Also recall that the dilation operator\ $X\longmapsto c\cdot X$ is defined
by $C(t;c\cdot X)=C(ct;X),$ for those $c>0$ for which $C(ct;X)$ is an FCGF.
The reflection and dilation operations do not commute, so that in general $%
c\cdot \left( -X\right) $ (if it exists) is different from $-c\cdot X$. We
hence define the \emph{M-transformation} by 
\begin{equation*}
X_{a}=c^{-1}\cdot \left[ -c\cdot \left( -X\right) \right]
\end{equation*}%
where $a=\left( 1-c\right) /c>-1$. The corresponding FCGF is%
\begin{equation*}
C(t;X_{a})=C\left( \frac{t}{1+at};X\right) \text{.}
\end{equation*}%
The inverse M-transformation is defined by%
\begin{equation*}
X_{a}=-c^{-1}\cdot \left( -c\cdot X\right) \text{,}
\end{equation*}%
where $a=\left( c-1\right) /c<1$.

The first and second factorial cumulants for $X_{a}$ are%
\begin{equation*}
\mathrm{E}\left( X_{a}\right) =\mathrm{E}\left( X\right) \text{ \qquad
and\qquad\ }\mathrm{S}\left( X_{a}\right) =\mathrm{S}\left( X\right) -2a%
\mathrm{E}\left( X\right) \text{.}
\end{equation*}%
It follows that the M-transformation may result in both overdispersion and
underdispersion, depending on whether $2a\mathrm{E}\left( X\right) $ is
smaller or bigger than $\mathrm{S}\left( X\right) $. The following example
shows a case where the M-transformation maps overdispersion into
underdispersion. Consider the negative binomial FCGF $-n\log \left( 1-\mu
t\right) $ (with integer $n$). For $a=\mu \in (0,1),$ the M-transformation
maps the negative binomial FCGF into the binomial FCGF $n\log \left( 1+\mu
t\right) $. The corresponding inverse transformation is obtained for $a=-\mu 
$.

The M-transformation hence provides one more tool in the study of
over/underdispersion, see also \cite{Kokonendji2008}, who used weighted
Poisson distributions for this purpose. An application of the
M-transformation is given in Section \ref{sec4.9}.

\section{Factorial tilting and factorial dispersion models\label{sec3}}

We now introduce a factorial tilting operator, similar to exponential
tilting, which leads to our main definitions of factorial tilting families
and factorial dispersion models, providing discrete analogues of
conventional exponential tilting and exponential dispersion models,
respectively, as well as to the geometric dispersion models of \cite%
{Jorgensen2010}. We also introduce the class of Poisson-Tweedie factorial
dispersion models, which provide a parallel with the Tweedie class of
exponential dispersion models, see 
\citet[Ch.~3--4]{Jorgensen1997}%
. In the following we use the notation of Section \ref{sec:general}.

\subsection{Factorial tilting families}

Consider the set $\overline{\mathcal{K}}$ of real analytic functions $K:%
\mathrm{dom}(K)\rightarrow \mathbb{R}$ satisfying $0\in \mathrm{dom}(K)$ and 
$K(0)=0$, where $\mathrm{dom}(K)$ denotes the largest interval containing
zero where $K$ is analytic. We define the \emph{tilting} of $K$ by the
amount $\theta \in \mathrm{dom}(K)$ as the function $K_{\theta }:\mathrm{dom}%
(K_{\theta })\rightarrow \mathbb{R}$ given by 
\begin{equation*}
K_{\theta }(t)=K(\theta +t)-K(\theta )\text{ for }t\in \mathrm{dom}%
(K_{\theta })=\mathrm{dom}(K)-\theta \text{.}
\end{equation*}%
The tilting operator defines an equivalence relation on $\overline{\mathcal{K%
}}$. In particular, if $\kappa \in \mathcal{K}$, then $\kappa _{\theta }$ is
the conventional exponential tilting of $\kappa $ 
\citep[cf.][p.~43]{Jorgensen1997}%
. If we restrict the tilting operator to $\mathcal{K}$, the corresponding
set of equivalence classes form the class of natural exponential families,
i.e. CGF families of the form $\left\{ \kappa _{\theta }\in \mathcal{K}%
:\theta \in \mathrm{dom}(\kappa )\right\} $ for given $\kappa \in \mathcal{K}
$. The corresponding natural exponential family has PDFs of the form 
\begin{equation*}
f(x;\theta )=g(x)\exp \left[ \theta x-\kappa (\theta )\right] \text{ for }%
\theta \in \mathrm{dom}(\kappa )
\end{equation*}%
with respect to a suitable dominating measure, where $g$ is the PDF
corresponding to $\kappa $.

Let us now instead consider the restriction of the tilting operator to the
class of FCGFs $\mathcal{C}$. We call this the \emph{factorial tilting
operator}. The corresponding set of equivalence classes in $\mathcal{C}$ are
called \emph{factorial tilting families}, namely FCGF families of the form 
\begin{equation*}
\left\{ C_{\theta }\in \mathcal{C}:\theta \in \mathrm{dom}(C)\right\}
\end{equation*}%
for given $C\in \mathcal{C}$. Note that when $C\in \mathcal{C}$, $\mathrm{dom%
}(C)$ is restricted to the interval $t>-1$. The distribution with FCGF $%
C_{\theta }$ has mean $\mu =\dot{C}(\theta ),$ and dispersion $\ddot{C}%
(\theta )$.

The factorial and exponential tilting operators turn out to be related by
means of dilation. Thus, for given $\theta =e^{\phi }-1\in \mathrm{dom}(C)$,
and for $C(t)=\kappa (\log \left( 1+t\right) ;X)$ as in (\ref{see}), we
obtain 
\begin{eqnarray}
C_{\theta }(t) &=&C(\theta +t)-C(\theta )  \notag \\
&=&\kappa (\log \left( 1+\theta +t\right) )-\kappa (\log \left( 1+\theta
\right) )  \notag \\
&=&\kappa \left( \log \left( 1+\frac{t}{1+\theta }\right) +\log \left(
1+\theta \right) \right) -\kappa (\log \left( 1+\theta \right) )  \notag \\
&=&\kappa _{\phi }\left( \log \left( 1+te^{-\phi }\right) \right) \text{.}
\label{diller}
\end{eqnarray}%
The form (\ref{diller}) is an exponential tilting of $\kappa $, followed by
a dilation. Conversely, the exponential tilting $\kappa _{\phi }$
corresponds to the FCGF 
\begin{equation*}
\kappa _{\phi }(\log \left( 1+t\right) )=C_{\theta }(te^{\phi }),
\end{equation*}%
which is a factorial tilting followed by a dilation.

In the special case of binomial thinning, we now derive the corresponding
expression for the probability mass function (PMF) of a discrete model. If $%
f $ is a given PMF, then the binomial thinning by $c$ has PMF%
\begin{equation*}
f_{c}(x)=\sum_{i=x}^{\infty }f(i)\tbinom{i}{x}c^{x}(1-c)^{i-x}\text{.}
\end{equation*}%
Now the exponential tilting by $\log (1+\theta )$ has density%
\begin{equation*}
f(x;\theta )=g(x)\left( 1+\theta \right) ^{x}e^{-C(\theta )}\text{.}
\end{equation*}%
Now take $c=1/(1+\theta )$ so that $1-c=\theta /(1+\theta ).$ The binomial
thinning by $c$ is then 
\begin{eqnarray*}
f_{c}(x;\theta ) &=&\sum_{i=x}^{\infty }g(i)\left( 1+\theta \right) ^{i}%
\tbinom{i}{x}\left( 1+\theta \right) ^{-x}\left[ \theta /(1+\theta )\right]
^{i-x}e^{-C(\theta )} \\
&=&\sum_{i=x}^{\infty }g(i)\tbinom{i}{x}\theta ^{i-x}e^{-C(\theta )}\text{.}
\end{eqnarray*}%
This is the PMF of the binomial thinning of the exponential tilting for $%
\theta >0$.

\subsection{Dispersion functions}

For a natural exponential family generated from the CGF $\kappa $, the
variance function $V=\ddot{\kappa}\circ \dot{\kappa}^{-1}$ is known to be a
useful characterization and convergence tool. We now introduce the
dispersion function for factorial tilting families, and show that it has
similar properties.

Let the FCGF $C=\kappa \left( \log \left( 1+\cdot \right) \right) \in 
\mathcal{C}$ be given, and let $\left\{ C_{\theta }:\theta \in \mathrm{dom}%
(C)\right\} $ be the factorial tilting family generated by $C$. All
factorial cumulants of $C_{\theta }$ are finite for $\theta \in \mathrm{int}%
\left( \mathrm{dom}(C)\right) $, the first two being the mean%
\begin{equation*}
\mu =\dot{C}_{\theta }(0)=\dot{C}(\theta )=\dot{\kappa}(\log \left( \theta
+1\right) )/\left( \theta +1\right)
\end{equation*}%
and the dispersion 
\begin{equation*}
\ddot{C}_{\theta }(0)=\ddot{C}(\theta )=\frac{\ddot{\kappa}(\log \left(
\theta +1\right) )-\dot{\kappa}(\log \left( \theta +1\right) )}{\left(
\theta +1\right) ^{2}}.
\end{equation*}%
Let $\Theta _{0}\subseteq \mathrm{dom}(C)$ be a non-degenerate interval
where $\ddot{C}(\theta )$ has constant sign, such that $\dot{C}(\theta )$ is
strictly monotone on $\Theta _{0}$, with $\mu =\dot{C}(\theta )$ belonging
to the interval $\Psi _{0}=\dot{C}(\Theta _{0})$. Here we define $\mu $ by
continuity at any end-point of $\Theta _{0}$ contained in $\Theta _{0}$ \cite%
[p.~46]{Jorgensen1997}, allowing infinite values of $\mu $, if necessary. We
say that the family is \emph{locally overdispersed }or\emph{\ locally
underdispersed} on $\Theta _{0}$, depending on the sign of $\ddot{C}(\theta
) $. We may then parametrize the family locally by the mean $\mu $, and we
denote the corresponding family member by $\mathrm{FT}(\mu )$. For a
globally overdispersed or underdispersed family, we may parametrize the
family globally by $\mu \in \Psi =\dot{C}(\mathrm{dom}(C))$. We adopt the
convention that for each $\mu \geq 0$, the Poisson distribution $\mathrm{Po}%
(\mu )$ forms an equidispersed factorial tilting family.

\begin{theorem}
\label{vtoc}Consider a locally overdispersed (underdispersed) factorial
tilting family and define the \emph{local dispersion function} $v:\Psi
_{0}\rightarrow \mathbb{R}$ by 
\begin{equation}
v(\mu )=\ddot{C}\circ \dot{C}^{-1}(\mu )\text{ for }\mu \in \Psi _{0}\text{,}
\label{defv}
\end{equation}%
where $v$ is defined by continuity at endpoints of $\Psi _{0}$ belonging to $%
\Psi _{0}$, and where $v(\mu )$ is positive (negative) for all $\mu \in \Psi
_{0}$. Then $v$ characterizes the family among all factorial tilting
families.
\end{theorem}

\begin{proof}
The proof is similar to the proof that a natural exponential family is
characterized by its variance function \cite[p.~51]{Jorgensen1997}. We first
show that the dispersion function does not depend on the choice of the FCGF $%
C$ representing the family. Thus, for given $\theta \in \Theta _{0}$, let us
derive the local dispersion function corresponding to $C_{\theta _{0}}$. For 
$t\in \mathrm{dom}(C)-\theta $ we obtain $\dot{C}_{\theta }(t)=\dot{C}%
(\theta +t),$ so that $\dot{C}_{\theta }(\Theta _{0}-\theta )=\dot{C}(\Theta
_{0})=\Psi _{0}$. The second derivative is $\ddot{C}_{\theta }(t)=\ddot{C}%
(\theta +t)$, and hence 
\begin{equation*}
\ddot{C}_{\theta }\circ \dot{C}_{\theta }^{-1}(\mu )=\ddot{C}\circ \dot{C}%
^{-1}(\mu )=v(\mu )\text{ for }\mu \in \Psi _{0}\text{.}
\end{equation*}%
It follows that $C_{\theta _{0}}$ yields the same local dispersion function
as $C$, so that $v$ represents an intrinsic property of the family.
To see that $v$ characterizes the family among all factorial tilting
families, we derive an inversion formula for $v$, again similar to the
inversion formula for the variance function. If the FCGF $C$ satisfies (\ref%
{defv}), then $\dot{C}^{-1}$ satisfies the equation%
\begin{equation*}
\frac{d\dot{C}^{-1}}{d\mu }(\mu )=\frac{1}{\ddot{C}\circ \dot{C}^{-1}(\mu )}=%
\frac{1}{v(\mu )}\text{.}
\end{equation*}%
For given $v$, the set of solutions to this equation are of the form $\dot{C}%
^{-1}(\mu )-\theta $, where $-\theta $ is an arbitrary constant. By solving
the equation $t=\dot{C}^{-1}(\mu )-\theta $ with respect to $\mu $ we obtain 
$\mu =\dot{C}(\theta +t)$, and integration with respect to $t$ in turn
yields the function $C_{\theta }(t)=C(\theta +t)-C(\theta )$ satisfying the
initial condition $C_{\theta }(0)=0$. Since $C_{\theta }$ is an FCGF if and
only if $\theta \in \mathrm{dom}(C)$, we have thus recovered the factorial
tilting family generated by $C$, as desired.
\end{proof}

For a globally overdispersed or underdispersed family, we refer to $v$ as
simply the dispersion function. The fact that a factorial tilting family is
characterized by the relations between its first two factorial cumulants
provides an example of a family with finitely generated cumulants in the
sense of \cite{Pistone1999}. \cite{Khatri1959} provides an early example of
a characterization of this form. Note that positive $v$ means that all
members of the family are overdispersed, negative $v$ means that all members
are underdispersed, whereas zero $v$ characterizes the Poisson family. The
next result shows that many important factorial tilting families are Poisson
mixtures, and hence overdispersed.

\begin{proposition}
\label{samevefu}The family of Poisson mixtures (\ref{mixexpo}) generated
from a natural exponential family with variance function $V$ yields an
overdispersed factorial tilting family with dispersion function $v=V$.
\end{proposition}

\begin{proof}
Consider the natural exponential family of CGFs $\kappa _{\theta }$
generated from the CGF $\kappa \in \mathcal{K}$. In view of (\ref{mixexpo}),
this family of CGFs is identical to the family of FCGFs for the
corresponding Poisson mixtures, which hence form a factorial tilting family,
and which is overdispersed due to the convexity of $\kappa $. The dispersion
function of this family is identical to the variance function $\ddot{\kappa}%
\circ \dot{\kappa}^{-1}\ $of the natural exponential family.
\end{proof}

\begin{example}[Binomial and negative binomial distributions]
For each value of the convolution parameter $\lambda >0$, the negative
binomial FCGFs form a factorial tilting family,%
\begin{equation}
C_{\theta }(t)=-\lambda \log \left( 1-\mu t\right)  \label{nb}
\end{equation}%
where $\mu =1/\left( 1-\theta \right) $. The mean is $m=\lambda \mu $, and
the dispersion function is $v(m)=\lambda ^{-1}m^{2}$ for $m>0$. Similarly,
for each integer $n$, the binomial FCGFs also form a factorial tilting
family, 
\begin{equation}
C_{\theta }(t)=n\log (1+\mu t)  \label{bi}
\end{equation}%
where $\mu =1/\left( 1+\theta \right) $. The mean is $m=\lambda \mu $, and
the dispersion function is $v(m)=-n^{-1}m^{2}$ for $0<m<n$.
\end{example}

\begin{proposition}
\label{binb}The binomial, negative binomial and Poisson families are the
only factorial tilting families that are closed under binomial thinning.
\end{proposition}

\begin{proof}
Let $\mathrm{FT}(\mu )$ denote the factorial tilting family with dispersion
function $v(\mu )$. If the family is closed under binomial thinning, then $%
c\cdot \mathrm{FT}(\mu )=\mathrm{FT}(c\mu )$, and hence $c^{2}v(\mu )=v(c\mu
)$. Taking $m=c\mu $, this implies that $v(m)$ is either zero or
proportional to $m^{2}.$ By Theorem \ref{vtoc} this, in turn, implies that $%
\mathrm{FT}(\mu )$ is either one of the binomial, negative binomial or
Poisson families.
\end{proof}

In view of (\ref{diller}), we conclude that a factorial tilting family that
is at the same time a natural exponential family must be closed under
binomial thinning, and is hence either binomial or negative binomial. The
Poisson natural exponential family is not included here, because each
Poisson distribution $\mathrm{Po}(\mu )$ is, on its own, a factorial tilting
family.

\subsection{Factorial dispersion models}

We now introduce factorial dispersion models as two-parameter families of
FCGFs obtained by combining the operations of factorial tilting and
convolution/division. For given $C\in \mathcal{C}$ and $\lambda >0$ we
consider the following type of FCGF \emph{(additive case):}%
\begin{equation}
t\longmapsto \lambda C_{\theta }(t)=\lambda C(\theta +t)-\lambda C(\theta )
\label{svd1}
\end{equation}%
for $\theta \in \Theta _{0}$. The second expression of (\ref{svd1}) shows
that the domain for $(\theta ,\lambda )$ is a product set $\Theta _{0}\times
\Lambda $, with $\Lambda =\mathbb{R}_{+}$ if $C$ is infinitely divisible. If 
$C$ is not infinitely divisible, the domain $\Lambda $ is a subset of $%
\mathbb{R}_{+}$ containing $\mathbb{N}$.

Like for exponential dispersion models, it is useful to consider as well the 
\emph{reproductive case}, obtained by the dilation $\gamma =1/\lambda $,
which yields the FCGF 
\begin{equation}
t\longmapsto \gamma ^{-1}C_{\theta }(\gamma t)\text{,}  \label{svd2}
\end{equation}%
where the domain for the dispersion parameter $\gamma $ is restricted to
those values for which the dilation exists. We may parametrize a factorial
tilting family locally (but not necessarily globally) by the mean $\mu =\dot{%
C}(\theta )$ of (\ref{svd2}), in which case we denote the distributions
corresponding to (\ref{svd1}) and (\ref{svd2}) by $\mathrm{FD}^{\ast }(\mu
,\lambda )$ and $\mathrm{FD}(\mu ,\gamma )$, respectively. The dilation $%
\mathrm{FD}(\mu ,\gamma )=\gamma \cdot \mathrm{FD}^{\ast }(\mu ,\gamma
^{-1}) $ linking the two cases is called the \emph{duality transformation}.
The following table summarizes the two types of factorial dispersion models.%
\begin{equation*}
\begin{tabular}{|l|c|c|c|c|}
\hline
Type & Symbol & FCGF & Mean & Dispersion \\ \hline
Additive & $\mathrm{FD}^{\ast }(\mu ,\lambda )$ & $\lambda C_{\theta }(t)$ & 
$\lambda \mu $ & $\lambda v(\mu )$ \\ \hline
Reproductive & $\mathrm{FD}(\mu ,\gamma )$ & $\gamma ^{-1}C_{\theta }(\gamma
t)$ & $\mu $ & $\gamma v(\mu )$ \\ \hline
\end{tabular}%
\end{equation*}%
We note in passing that the zero-inflation index $\mathrm{ZI}\left( X\right)
=1+C_{\theta }(-1;X)/\dot{C}_{\theta }(0;X)$ does not depend on the value of 
$\lambda $ in the additive case, but only on $\mu $.

For a factorial dispersion model generated by $C$, we refer to $C$ and $v$
as the unit FCGF and unit dispersion function, respectively. The additive
form $\mathrm{FD}(\mu ,\gamma )$ is often useful because of its simple
dispersion function $\gamma v(\mu )$, whereas $\mathrm{FD}^{\ast }(\mu
,\lambda ),$ with mean $m=\lambda \mu ,$ say, has dispersion function $%
m\longmapsto \lambda v(m/\lambda ).$ An additive factorial tilting family $%
\mathrm{FD}^{\ast }(\mu ,\lambda )$ is closed under convolution, 
\begin{equation*}
\mathrm{FD}^{\ast }(\mu ,\lambda _{1})+\mathrm{FD}^{\ast }(\mu ,\lambda
_{2})=\mathrm{FD}^{\ast }(\mu ,\lambda _{1}+\lambda _{2})\text{.}
\end{equation*}%
\ For $Y_{1},\ldots ,Y_{n}$ i.i.d. $\mathrm{FD}(\mu ,\gamma )$ then the
dilation average $\overline{Y}_{n}=n^{-1}\cdot \left( Y_{1}+\cdots
+Y_{n}\right) $ satisfies the following reproductive property:%
\begin{equation}
\overline{Y}_{n}\sim \mathrm{FD}(\mu ,\gamma /n)\text{.}  \label{average}
\end{equation}

The additive binomial and negative binomial factorial dispersion models are
apparent from (\ref{bi}) and (\ref{nb}), respectively. The corresponding
reproductive FCGFs take the form 
\begin{equation}
t\mapsto \pm \gamma ^{-1}\log \left( 1\pm \gamma \mu t\right) ,
\label{ExgeoP}
\end{equation}%
where $\gamma =1/n$ or $\gamma =1/\lambda $, respectively, which correspond
to reparametrization of the two models in terms of the mean $\mu $ and the
dispersion parameter $\gamma $.

\subsection{Poisson-Tweedie mixtures and power dispersion functions\label%
{sec3.6}}

We have already introduced Poisson mixtures in Section \ref{sec:Mixtures},
and we now consider the class of Poisson-Tweedie mixtures (Hougaard \emph{et
al.}, 1997; El-Shaarawi \emph{et al.}, 2011\nocite{El-Shaarawi2011}), which
are in many ways analogous to ordinary Tweedie models, and includes several
well-known distributions as special cases.

Consider the Tweedie exponential dispersion model $\mathrm{Tw}_{p}(\mu
,\gamma )$, which has mean $\mu \in \Omega _{p}$, dispersion parameter $%
\gamma >0$, and unit variance function%
\begin{equation}
V(\mu )=\mu ^{p}\qquad \text{for }\quad \mu \in \Omega _{p}\text{,}
\label{disp}
\end{equation}%
where $p\notin \left( 0,1\right) $, $\Omega _{0}=\mathbb{R},$ and $\Omega
_{p}=\mathbb{R}_{+}$ for $p\neq 0.$ The Poisson-Tweedie mixture $\mathrm{PT}%
_{p}(\mu ,\gamma )$ is defined as the Poisson mixture $\mathrm{PT}_{p}(\mu
,\gamma )=P\left( \mathrm{Tw}_{p}(\mu ,\gamma );1\right) $. Here we require
that $p\geq 1$, in order to make $\mathrm{Tw}_{p}(\mu ,\gamma )$
non-negative. For each $p\geq 1$, the Poisson-Tweedie mixture $Y\sim \mathrm{%
PT}_{p}(\mu ,\gamma )$ is an overdispersed factorial dispersion model with
mean $\mu $, unit dispersion function $v(\mu )=\mu ^{p}$ defined by (\ref%
{disp}), and variance 
\begin{equation}
\mathrm{Var}\left( Y\right) =\mu +\gamma \mu ^{p}\text{.}  \label{Poispow}
\end{equation}%
The Poisson-Tweedie mixture $\mathrm{PT}_{p}(\mu ,\gamma )$ satisfies the
following dilation property:%
\begin{equation}
c\cdot \mathrm{PT}_{p}(\mu ,\gamma )=\mathrm{PT}_{p}(c\mu ,c^{2-p}\gamma
)\qquad \text{for }\quad c>0\text{.}  \label{dilation}
\end{equation}%
In the following, we use the notation $\mathrm{PT}_{p}(\mu ,\gamma )$ for
any factorial dispersion model with power dispersion function, even if it is
not a Poisson-Tweedie mixture.

The next theorem presents a characterization of factorial dispersion models
that satisfy a dilation property like (\ref{dilation}); similar to the
characterization theorem for Tweedie exponential dispersion models \cite[%
p.~128]{Jorgensen1997}. Table \ref{tablep} summarizes the main types of
factorial dispersion models with power dispersion functions, including the
Hermite and the Poisson-binomial distributions, which are not
Poisson-Tweedie mixtures. Other values of $p$ than those found in Table \ref%
{tablep} are possible, as shown in Example \ref{COM-Poisson} below.

\begin{theorem}
\label{characterize}Let $\mathrm{FD}(\mu ,\gamma )$ be a non-degenerate
locally overdispersed factorial dispersion model satisfying $\inf \Psi
_{0}\leq 0$ or $\sup \Psi _{0}=\infty $, such that for some $\gamma >0$ and
an interval of $c$-values%
\begin{equation}
c^{-1}\cdot \mathrm{FD}(c\mu ,\varphi _{c}\gamma )=\mathrm{FD}(\mu ,\gamma )%
\text{ for }\mu \in \Psi _{0}\text{,}  \label{fixedpoint}
\end{equation}%
where $\varphi _{c}$ is a positive function of $c$. Then $\mathrm{FD}(\mu
,\gamma )$ has power dispersion function proportional to $\mu ^{p}$ for some 
$p\in \mathbb{R}$, and $\varphi _{c}=c^{2-p}$.
\end{theorem}

\begin{proof}
Calculating the dispersion on each side of (\ref{fixedpoint}) gives for an
interval of $c$-values 
\begin{equation}
c^{-2}\varphi _{c}\gamma v(c\mu )=\gamma V(\mu )\text{ for }\mu \in \Psi _{0}%
\text{,}  \label{functional}
\end{equation}%
where $v$ is the local unit dispersion function of $\mathrm{FD}(\mu ,\gamma )
$. Taking, without loss of generality, $\mu =1$ in (\ref{functional}) gives $%
\varphi _{c}=c^{2}V(1)/V(c)$, which together with (\ref{functional}) implies
that $v$ satisfies the functional equation $v(1)v(c\mu )=v(c)v(\mu )$. 
This equation is equivalent to Cauchy's functional equation. By
the continuity of $v$, the solutions to this equation are of the form $v(\mu
)=\lambda \mu ^{p}$ for some $p\in \mathbb{R}$ and $\lambda \neq 0$ because
the family is non-degenerate (i.e. non-Poisson). This, in turn, implies that 
$\varphi _{c}=c^{2-p}$. In view of Theorem \ref{vtoc}, $\mathrm{FD}(\mu
,\gamma )$ is hence a Poisson-Tweedie model in the case $p\geq 1$. For
values of $p$ less than $1$, the model $\mathrm{FD}(\mu ,\gamma )$, if it
exists, is not a Poisson-Tweedie mixture. 
\end{proof}

\begin{table}[tbp] \centering%
\caption{The main types of factorial dispersion models with power dispersion functions.}%
\begin{tabular}{|l|c|c|}
\hline
\textbf{Type} & $p$ & $\alpha $ \\ \hline
Hermite & $p=0$ & $\alpha =2$ \\ 
Poisson-binomial & $p=(n-2)/(n-1)$ & $n=3,4,\ldots $ \\ 
Neyman Type A & $p=1$ & $\alpha =-\infty $ \\ 
Poisson-negative binomial & $1<p<1$ & $\alpha <0$ \\ 
Pólya-Aeppli & $p=3/2$ & $\alpha =-1$ \\ 
Negative binomial/binomial & $p=2$ & $\alpha =0$ \\ 
Factorial discrete stable & $p>2$ & $0<\alpha <1$ \\ 
Poisson-inverse Gaussian & $p=3$ & $\alpha =1/2$ \\ \hline
\end{tabular}%
\label{tablep}%
\end{table}%

\subsection{Discrete stable factorial dispersion models\label{sec4.4}}

The case $p>2$ in Table \ref{tablep} correspond to factorial dispersion
models generated by discrete $\alpha $-stable distributions with $\alpha \in
(0,1)$, where $\alpha $ is defined from the power parameter $p$ by

\begin{equation}
\alpha =1+(1-p)^{-1}\text{,}  \label{alfap}
\end{equation}%
with the convention that $\alpha =-\infty $ for $p=1$ 
\citep[p.~131]{Jorgensen1997}%
. In particular, the case $p=3$ corresponds to the "Sichel" or
Poisson-inverse Gaussian distributions with unit dispersion function $v(\mu
)=\mu ^{3}$; see \cite{Willmot1987}.

The discrete $\alpha $-stable distribution with $\alpha \in \left(
0,1\right) $ was introduced by Steutel \& van Harn, 1979)\nocite{Steutel1979}%
, and corresponds to FCGFs proportional to 
\begin{equation*}
C^{(\alpha )}(t)=\frac{\alpha -1}{\alpha }\left( \frac{t}{\alpha -1}\right)
^{\alpha }\text{ for }t/\left( \alpha -1\right) >0\text{.}
\end{equation*}%
Factorial tilting and infinite division/convolution with power $\lambda >0$
yield the Poisson-Tweedie mixture $\mathrm{PT}_{p}^{\ast }(\mu ,\lambda )$
(additive version) with FCGF 
\begin{equation}
\lambda C_{\theta }^{(\alpha )}(t)=\lambda C^{(\alpha )}(\theta )\left[
\left( 1+t/\theta \right) ^{\alpha }-1\right] ,  \label{Tweedie}
\end{equation}%
where the parameter $\mu $ is defined by 
\begin{equation}
\mu =\dot{C}^{(\alpha )}(\theta )=\left( \frac{\theta }{\alpha -1}\right)
^{\alpha -1}\text{.}  \label{bjrm:muthe}
\end{equation}%
An application of the duality transform then yields the Poisson-Tweedie
mixture $\mathrm{PT}_{p}(\mu ,\gamma )$ with $\gamma =1/\lambda $.

This construction of the Poisson-Tweedie mixtures is analogous to the
construction of the Tweedie model $\mathrm{Tw}_{p}(\mu ,\gamma )$ as an
exponential tilting of a positive $\alpha $-stable distribution in the case $%
\alpha \in (0,1)$. It is important to emphasize, however, that the above
results could not have been easily obtained by means of exponential tilting.
To illustrate this point, we note that the Poisson-inverse Gaussian mixture,
when considered as an exponential dispersion model, has unit variance
function given by 
\begin{equation*}
V(\mu )=\mu +\frac{\mu ^{3}}{2}+\frac{\mu ^{2}}{2}\sqrt{2+\mu ^{2}}\text{
for }\mu >0\text{,}
\end{equation*}%
as compared with the variance $\mu +\gamma \mu ^{3}$ obtained from (\ref%
{Poispow}). For general $p\geq 1$, the Poisson-Tweedie exponential
dispersion models have unit variance functions of the form%
\begin{equation*}
V(\mu )=\mu +\mu ^{p}\exp \left[ (2-p)H(\mu )\right]
\end{equation*}%
where $H(\mu )$ is implicitly defined (Kokonendji \emph{et al.}, 2004;\nocite%
{Kokonendji2004} Jørgensen, 1997),\nocite{Jorgensen1997} in sharp contrast
to (\ref{Poispow}). We also note that the so-called Hinde-Demétrio\emph{\ }%
class of exponential dispersion models have unit variance functions of the
form 
\begin{equation*}
V(\mu )=\mu +\mu ^{p}\text{,}
\end{equation*}%
but are not in general integer-valued (Kokonendji \emph{et al.}, 2004).%
\nocite{Kokonendji2004}

\subsection{Poisson-binomial and Poisson-negative binomial distributions 
\label{sec4.9}}

Consider the Poisson-negative binomial FCGF 
\citep[p. 414]{Johnson2005}%
, defined by 
\begin{equation}
C(t)=\lambda \left[ \left( 1-\mu t\right) ^{-k}-1\right] ,  \label{Poly}
\end{equation}%
where $\lambda >0$, which is essentially of the form (\ref{Tweedie}) with $%
\alpha =-k<0,$ corresponding to $1<p<2$. The case $\lambda =-1$ ($p=3/2$) is
the Pólya-Aeppli distribution, and (\ref{Poly}) is also known as a
generalized Pólya-Aeppli distribution.

Similarly, let us consider the Poisson-binomial FCGF (Johnson \emph{et al.},
2005, p. 401),\nocite{Johnson2005} defined by 
\begin{equation*}
C(t)=\lambda \left[ \left( 1+\mu t\right) ^{n}-1\right] \text{,}
\end{equation*}%
where $\lambda >0$ and $n\in \mathbb{N}$. Up to a reparametrization, this
FCGF is of the form (\ref{Tweedie}) with $\alpha =n$, corresponding to $%
p=(n-2)/(n-1)\in (0,1)$ for $n\geq 2$, which are not Poisson-Tweedie
mixtures. The case $n=1$ gives the Poisson distribution, whereas $n=2$ gives
the Hermite distribution. The Poisson-binomial distribution satisfies a
binomial thinning property like (\ref{dilation}) for $c\in (0,1)$.

It is not immediately clear if there exist factorial dispersion models with $%
p\in (0,1)$ corresponding to non-integer values of $\alpha >2$. The
following, formal considerations suggest that the answer to this question
may be affirmative. To this end, consider the M-transformation of the model $%
\mathrm{PT}_{p}^{\ast }(\mu ,\lambda )$ with $a=-1/\theta $, which has FCGF%
\begin{eqnarray*}
C_{\theta }^{(\alpha )}\left( \frac{t}{1+at}\right) &=&C^{(\alpha )}(\theta )%
\left[ \left( 1+\frac{t/\theta }{1-t/\theta }\right) ^{\alpha }-1\right] \\
&=&C^{(\alpha )}(\theta )\left[ \left( 1-t/\theta \right) ^{-\alpha }-1%
\right] \text{.}
\end{eqnarray*}%
This FCGF is proportional to $C_{-\theta }^{(-\alpha )}(t),$ provided that
the following ratio is positive:

\begin{equation*}
\frac{C^{(\alpha )}(\theta )}{C^{(-\alpha )}(-\theta )}\propto \left( \alpha
-1\right) /\left( \alpha +1\right) \text{,}
\end{equation*}%
which is the case for $\left\vert \alpha \right\vert >1$. In particular, the
set $1<p<4/3$ ($\alpha <-2$) is mapped onto the set $0<p<1$ ($\alpha >2$).
Similarly, the set $4/3<p<3/2$ ($-2<\alpha <-1$) is mapped onto the set $p<0$
($1<\alpha <2$). The existence of the corresponding factorial dispersion
models will be shown in Example \ref{COM-Poisson} below.

\subsection{Neyman Type A distribution\label{Pss}}

The Neyman Type A distribution $\mathrm{PT}_{1}(\mu ,\gamma )$ is a Poisson
mixture of Poisson distributions, corresponding to the FCGF 
\begin{equation*}
C(t)=\gamma ^{-1}\mu \left( e^{\gamma t}-1\right) \text{,}
\end{equation*}%
see for example Dobbie \& Welsh (2001)\nocite{Dobbie2001} and Massé \&
Theodorescu (2005).\nocite{Masse2005} The variance of $Y\sim \mathrm{PT}%
_{1}(\mu ,\gamma )$ is 
\begin{equation*}
\mathrm{Var}\left( Y\right) =\mu \left( 1+\gamma \right) \text{,}
\end{equation*}%
which is special by not being asymptotic to $\mu $ near zero as is the case
for Poisson-Tweedie mixtures with $p>1$.

Like all reproductive Poisson-Tweedie mixtures, the parameter vector $(\mu
,\gamma )$ is identifiable from the distribution $\mathrm{PT}_{1}(\mu
,\gamma )$, by means of the first two factorial cumulants $\mu $ and $\gamma
v(\mu )$, similar to the case of reproductive exponential dispersion models.
This is not, however, the case for the parameter $(\mu ,\lambda )$ of the
additive factorial dispersion model $\mathrm{PT}_{1}^{\ast }(\mu ,\lambda )$
with FCGF 
\begin{equation*}
C(t)=\lambda \mu \left( e^{t}-1\right) ,
\end{equation*}%
where only the mean $\lambda \mu $ is identifiable. The following result
shows that this is essentially the only additive factorial dispersion model
with this defect.

\begin{theorem}
\label{Linear}Consider a locally overdispersed or underdispersed additive
factorial dispersion model $\mathrm{FD}^{\ast }(\mu ,\lambda )$. If the
factorial tilting families $\mathrm{FD}^{\ast }(\cdot ,\lambda )$ are
identical for an interval of $\lambda $-values, then $\mathrm{FD}^{\ast
}(\mu ,\lambda )$ is a Neyman Type A family.
\end{theorem}

\begin{proof}
We can assume, without loss of generality, that $(1,1)\in \Psi _{0}\times
\Lambda $, the domain for $(\mu ,\lambda )$. Let $v$ denote the unit
dispersion function of $\mathrm{FD}^{\ast }(\mu ,\lambda )$. If the
factorial tilting family $\mathrm{FD}^{\ast }(\cdot ,1)$ is identical to the
family $\mathrm{FD}^{\ast }(\cdot ,\lambda )$, then the two local dispersion
functions are identical, i.e. $\lambda v(m/\lambda )=v(m)$, which for $m=1$
implies $v(1/\lambda )=v(1)/\lambda $ for an interval of $\lambda $-values.
We conclude that $v\left( \mu \right) $ is proportional to $\mu $, which in
view of Theorem \ref{vtoc} implies that $\mathrm{FD}^{\ast }(\mu ,\lambda )$
is a Neyman Type A family.
\end{proof}

The situation is hence analogous to the case of additive exponential
dispersion models, among which only the scaled Poisson family has this lack
of identifiability 
\citep[p.~74]{Jorgensen1997}%
.

\section{Power asymptotics and Poisson-Tweedie convergence\label{sec5}}

We now consider power asymptotics for dispersion functions of factorial
dispersion models, which proves convergence to distributions in the class of
Poisson-Tweedie mixtures, similar to the Tweedie convergence theorem of 
\citet{Jorgensen1994}%
, see also 
\citet[][Ch.~4]{Jorgensen1997}%
. This approach provides a unified method of proof for a range of different
convergence results for discrete distributions, many of which are new.

\subsection{Convergence of dispersion functions}

We first present a general convergence theorem for factorial tilting
families, which is used for proving the Poisson-Tweedie convergence theorem
below (Theorem \ref{converges}). The result is similar to the \cite{Mora1990}
convergence theorem for variance functions \cite[p.~54]{Jorgensen1997},
which says that convergence of a sequence of variance functions, when the
convergence is uniform on compact sets, implies weak convergence of the
corresponding sequence of natural exponential families.

\begin{theorem}
\label{Moraq}Let $\left\{ \mathrm{FT}_{n}(\mu ):n=1,2,\ldots \right\} $
denote a sequence of locally overdispersed or underdispersed factorial
tilting families having local dispersion functions $v_{n}$ with domains $%
\Psi _{n}$. Suppose that

\begin{enumerate}
\item $\bigcap\limits_{n=1}^{\infty }\Psi _{n}$ contains a non-empty
interval $\Psi _{0}$;

\item $\lim_{n\rightarrow \infty }v_{n}(\mu )=v(\mu )$ exists uniformly on
compact subsets of $\mathrm{int}\Psi _{0}$;

\item $v(\mu )\neq 0$ for all $\mu \in \mathrm{int}\Psi _{0}$ or $v(\mu )=0$
for all $\mu \in \mathrm{int}\Psi _{0}$.
\end{enumerate}

\noindent In the case $v(\mu )\neq 0$, there exists a factorial tilting
family $\mathrm{FT}(\mu )$ whose local dispersion function coincides with $v$
on $\mathrm{int}\Psi _{0}$, such that for each $\mu $ in $\mathrm{int}\Psi
_{0}$ the sequence of distributions $\mathrm{FT}_{n}(\mu )$ converges weakly
to $\mathrm{FT}(\mu )$. In the case $v(\mu )=0$, $\mathrm{FT}_{n}(\mu )$
converges weakly for each $\mu $ in $\mathrm{int}\Psi _{0}$ to the Poisson
distribution $\mathrm{Po}(\mu )$.
\end{theorem}

The proof of Theorem \ref{Moraq}, which is given in Appendix B, is similar
to the proof by 
\citet{Mora1990}%
, see also 
\citet[p.~54]{Jorgensen1997}%
. The case of convergence to a zero dispersion function follows the same
line of proof as in \cite{Jorgensen2010} for geometric dispersion models.

We now use Theorem \ref{Moraq} to give a new proof of the Poisson law of
thin numbers (Proposition \ref{LTN}). Let us first note that a locally
overdispersed or underdispersed reproductive factorial dispersion model $%
\mathrm{FD}(\mu ,\gamma )$ has local dispersion function of the form $\gamma
v(\mu )$, which goes to zero as $\gamma \downarrow 0$. It is easy to show
that the limit exists uniformly on compact subsets of $\Psi _{0}$. By
Theorem \ref{Moraq} this implies 
\begin{equation}
\mathrm{FD}(\mu ,\gamma )\overset{D}{\rightarrow }\mathrm{Po}(\mu )\text{ as 
}\gamma \downarrow 0\text{.}  \label{georen}
\end{equation}%
This applies, in particular, to all Poisson-Tweedie mixtures and power
dispersion function models $\mathrm{PT}_{p}(\mu ,\gamma )$. The result
implies that all factorial dispersion models resemble the Poisson
distribution for small dispersion, irrespective of their origin.
Furthermore, consider the dilation average $\overline{Y}_{n}$ based on $%
Y_{1},\ldots ,Y_{n}$ i.i.d. from $\mathrm{FD}(\mu ,\gamma )$, which, by (\ref%
{average}), has distribution $\overline{Y}_{n}\sim \mathrm{FD}(\mu ,\gamma
/n)$. By (\ref{georen}) this implies that 
\begin{equation*}
\overline{Y}_{n}\overset{D}{\rightarrow }\mathrm{Po}(\mu )\text{ as }%
n\rightarrow \infty \text{.}
\end{equation*}%
We have hence obtained a new proof of the law of thin numbers. This Poisson
convergence result is analogous to the exponential convergence result for
geometric dispersion models of \cite{Jorgensen2010}.

\subsection{Power asymptotics}

To motivate the next Poisson-Tweedie convergence theorem, let us rewrite the
dilation result (\ref{dilation}) in the form of a fixed point 
\begin{equation*}
c^{-1}\cdot \mathrm{PT}_{p}(c\mu ,c^{2-p}\gamma )=\text{$\mathrm{PT}$}%
_{p}(\mu ,\gamma )\text{.}
\end{equation*}%
The next theorem shows that this fixed point has a domain of attraction
characterized by a power asymptotic dispersion function. The theorem is
analogous to the Tweedie convergence theorem for exponential dispersion
models 
\citep[pp.~148--149]{Jorgensen1997}
and to similar convergence results for extreme and geometric dispersion
models 
\citep{Jorgensen2007a,Jorgensen2010}%
.

\begin{theorem}
\label{converges}Let $\mathrm{FD}(\mu ,\gamma )$ denote a locally
overdispersed or underdispersed factorial dispersion model with unit
dispersion function $v$ on $\Psi _{0}$, such that either $\inf \Psi _{0}\leq
0$ or $\sup \Psi _{0}=\infty $. Assume that for some $p\in \mathbb{R}$ the
unit dispersion function satisfies $v(\mu )\sim c_{0}\mu ^{p}$ as either $%
\mu \downarrow 0$ or $\mu \rightarrow \infty $. Then for each $\mu \in
\Omega _{p}$ 
\begin{equation}
c^{-1}\cdot \mathrm{FD}(c\mu ,c^{2-p}\gamma )\overset{D}{\longrightarrow }%
\mathrm{PT}_{p}(\mu ,\gamma c_{0})\qquad \text{ as }c\downarrow 0\text{ or }%
c\rightarrow \infty ,  \label{convergence}
\end{equation}%
respectively. In the case $c\downarrow 0$, the model $\mathrm{FD}(\mu
,\gamma )$ is required to be infinitely dilatable, and if $%
c^{2-p}\rightarrow \infty $ the model is required to be infinitely divisible.
\end{theorem}

\begin{proof} Without loss of generality we may take $c_0=1$.
We first note that for each given value of $\gamma $ and $c$, the left-hand
side of (\ref{convergence}) is a factorial tilting family with mean $\mu $,
provided that $c$ is small (large) enough for $c\mu $ to belong to $\Psi
_{0} $. The corresponding dispersion function satisfies 
\begin{equation*}
c^{-2}c^{2-p}\gamma v(c\mu )\rightarrow \gamma \mu ^{p}\text{ as }%
c\downarrow 0\text{ or }c\rightarrow \infty \text{,}
\end{equation*}%
respectively, and hence converges to the dispersion function of $\mathrm{PT}%
_{p}(\mu ,\gamma )$. To show that the convergence is uniform in $\mu $ on
compact subsets of $\Omega _{p}$, let us consider the case where $%
c\downarrow 0$ (the proof is similar in the case $c\rightarrow \infty $).
Let $0<M_{1}\leq \mu \leq M_{2}<\infty $ and $\epsilon >0$ be given, and let 
$c$ be small enough to make 
\begin{equation*}
\left\vert \frac{v(c\mu )}{\left( c\mu \right) ^{p}}-1\right\vert <\epsilon
\end{equation*}%
for all $\mu \leq M_{2}$. Then 
\begin{equation*}
\left\vert \frac{v(c\mu )}{c^{p}}-\mu ^{p}\right\vert =\mu ^{p}\left\vert 
\frac{v(c\mu )}{\left( c\mu \right) ^{p}}-1\right\vert \leq
(M_{1}^{p}+M_{2}^{p})\epsilon \text{,}
\end{equation*}%
which shows that the convergence is uniform on the compact interval $%
M_{1}\leq \mu \leq M_{2}$. The result (\ref{convergence}) now follows from
Theorem \ref{Moraq}.
\end{proof}

Many factorial dispersion models have power asymptotic dispersion functions,
and are hence asymptotically similar to Poisson-Tweedie mixtures. Thus,
under the hypothesis of Theorem \ref{converges}, the dilation property (\ref%
{dilation}) for $\mathrm{PT}_{p}(\mu ,\gamma )$ implies the following
distribution approximation: 
\begin{equation}
\mathtt{\mathrm{FD}}(c\mu ,c^{2-p}\gamma )\overset{\cdot }{\sim }\mathrm{P}%
\text{$\mathrm{T}$}_{p}(c\mu ,c^{2-p}\gamma c_{0})  \label{approx}
\end{equation}%
for $c$ small or $c$ large, respectively. In view of the fact that any FCGF
belongs to some factorial dispersion model (namely the model generated by
the FCGF itself), many factorial dispersion models may be approximated by
Poisson-Tweedie models in this way.

\begin{example}[Discrete Linnik distribution]
The discrete Linnik distribution is defined by the FCGF%
\begin{equation}
C(t)=-b\log \left[ 1+c\left( -t\right) ^{\alpha }\right]  \label{Linnik}
\end{equation}%
where $0<\alpha <1$ and $b,c>0$ 
\citep[p. 497]{Johnson2005}%
. This distribution is hence infinitely dilatable as well as infinitely
divisible. The asymptotic behaviour of $C(t)$ as $t\uparrow 0$ is 
\begin{equation*}
C(t)\sim -\lambda c\left( -t\right) ^{\alpha }\text{,}
\end{equation*}%
which, in turn, implies that the unit dispersion function is power
asymptotic at infinity, 
\begin{equation*}
v(\mu )\sim c_{0}\mu ^{p}\text{ as }\mu \rightarrow \infty
\end{equation*}%
for some $c_{0}>0$, where $p>2$ is related to $\alpha \in (0,1)$ by (\ref%
{alfap}). It follows that the factorial dispersion model generated by $C(t)$
satisfies (\ref{convergence}) as $c\rightarrow \infty $.
\end{example}

To connect the result (\ref{convergence}) with large sample theory, let $%
\bar{Y}_{n}\sim \mathrm{FD}(\mu ,\gamma /n)$ denote the dilation average of
an i.i.d. sample from the distribution $\mathrm{FD}(\mu ,\gamma )$ (cf. Eq. (%
\ref{average})). Then for $p\neq 2$ we may rewrite (\ref{convergence}) as
follows (taking $c^{2-p}=1/n$): 
\begin{equation}
n^{-1/(p-2)}\cdot \mathtt{\mathrm{FD}}(n^{1/(p-2)}\mu ,\gamma /n)\overset{D}{%
\rightarrow }\mathrm{PT}_{p}(\mu ,\gamma c_{0})\text{ }\qquad \text{as }%
n\rightarrow \infty \text{,}  \label{larsam}
\end{equation}%
so the scaled and factorially tilted dilation average $\bar{Y}_{n}$
converges to a Poisson-Tweedie model. We interpret this result via (\ref%
{approx}) as saying that a system subject to independent $\mathrm{FD}(\mu
,\gamma )$-distributed shocks will eventually settle in what may be called a
Poisson-Tweedie equilibrium.

Alternatively, let us consider the case where $\bar{Y}_{n}\sim \mathrm{FD}%
(\mu ,\gamma )$ is the dilation average of an i.i.d. sample from the
distribution $\mathrm{FD}(\mu ,n\gamma )$, which requires that the model $%
\mathrm{FD}(\mu ,\gamma )$ be infinitely divisible. Then for $p\neq 2$ we
may rewrite (\ref{convergence}) as follows (taking $c^{2-p}=n$): 
\begin{equation}
n^{-1/(2-p)}\cdot \mathtt{\mathrm{FD}}(n^{1/(2-p)}\mu ,\gamma n)\overset{D}{%
\rightarrow }\mathrm{PT}_{p}(\mu ,\gamma c_{0})\text{ }\qquad \text{as }%
n\rightarrow \infty \text{.}  \label{infdiv}
\end{equation}%
We interpret the result (\ref{infdiv}) as saying that the scaled and
factorially tilted component $\mathrm{FD}(\mu ,n\gamma )$ converges to a
Poisson-Tweedie model. The main feature of (\ref{infdiv}) is that the signs
of the powers of $n$ are reversed compared with (\ref{larsam}).

The results (\ref{larsam}) and (\ref{infdiv}) both require $p\neq 2$, which
together with Proposition \ref{binb} highlights the special role of the
binomial and negative binomial distributions in Poisson-Tweedie asymptotics,
as we shall now see.

\subsection{Binomial and negative binomial convergence\label{sec-gamma}}

We now discuss the power asymptotics of Theorem \ref{converges} in the case $%
p=2$ ($\alpha =0$). We first note that the dilation property (\ref%
{fixedpoint}) for the negative binomial distribution $\mathrm{PT}_{2}(\mu
,\gamma )$ takes the form 
\begin{equation*}
c^{-1}\cdot \mathrm{PT}_{2}(c\mu ,\gamma )=\mathrm{PT}_{2}(\mu ,\gamma )%
\text{ for }c>0\text{,}
\end{equation*}%
for all $\mu >0$ and $\gamma >0$. Suppose that the locally overdispersed
factorial dispersion model $\mathrm{FD}(\mu ,\gamma )$ with mean domain $%
\Psi _{0}$ is such that either $\inf \Psi _{0}\leq 0$ or $\sup \Psi
_{0}=\infty $, and assume that the unit dispersion function satisfies $v(\mu
)\sim \mu ^{2}$ as $\mu \downarrow 0$ or $\mu \rightarrow \infty $,
respectively. The corresponding version of (\ref{convergence}) is then 
\begin{equation}
c^{-1}\cdot \mathrm{FD}(c\mu ,\gamma )\overset{D}{\rightarrow }\mathrm{PT}%
_{2}(\mu ,\gamma )\text{ as }c\downarrow 0\text{ or }c\rightarrow \infty 
\text{,}  \label{gaco}
\end{equation}%
respectively, for all $\mu ,\gamma >0$. The result (\ref{gaco}) implies the
following negative binomial approximation: 
\begin{equation*}
\mathrm{FD}(c\mu ,\gamma )\overset{\cdot }{\sim }\mathrm{PT}_{2}(c\mu
,\gamma )\qquad \text{ as }c\downarrow 0\text{ or}\rightarrow \infty \text{,}
\end{equation*}%
respectively. The result does not involve a large sample in any sense, but
instead applies as the mean $c\mu $ goes to the boundary of the parameter
space.

An example is the factorial dispersion model $\mathrm{FD}(\mu ,\gamma )$
generated by the discrete Linnik distribution (\ref{Linnik}). The
corresponding unit dispersion function behaves as $c_{0}\mu ^{2}$ as $\mu
\downarrow 0$ for some $c_{0}>0$. It follows that (\ref{gaco}) is satisfied
as $c\downarrow 0$.

We now turn to the binomial distribution $\mathrm{Bi}(\mu ,n)$, which is a
factorial dispersion model on additive form. The binomial distribution
satisfies the following thinning property: 
\begin{equation*}
c^{-1}\cdot \mathrm{Bi}(c\mu ,n)=\mathrm{Bi}(\mu ,n)\text{ for }0<c<1\text{.}
\end{equation*}%
As a result, suppose that the underdispersed additive factorial dispersion
model $\mathrm{FD}^{\ast }(\mu ,n)$ has unit dispersion function satisfying $%
v(\mu )\sim -\mu ^{2}$ as $\mu \downarrow 0$. Then 
\begin{equation}
c^{-1}\cdot \mathrm{FD}^{\ast }(c\mu ,n)\overset{D}{\longrightarrow }\mathrm{%
Bi}(\mu ,n)\qquad \text{ as }c\downarrow 0\text{.}  \label{BiMo}
\end{equation}

\begin{example}[COM-Poisson distribution]
\label{COM-Poisson}Consider the COM-Poisson distribution (cf. Shmueli \emph{%
et al.}, 2005)%
\nocite{Shmueli2005}
with PMF%
\begin{equation}
P\left( X=x\right) =\frac{\lambda ^{x}}{\left( x!\right) ^{\nu }Z\left(
\lambda ,\nu \right) }\text{ for }x=0,1,\ldots \text{,}  \label{COM}
\end{equation}%
where $\lambda >0$, $\nu \geq 0$, and $Z\left( \lambda ,\nu \right) $ is a
normalizing constant. This family provides useful illustrations of several
of the above convergence results. The FCGF of (\ref{COM}) is 
\begin{equation}
C\left( t\right) =\log \frac{Z\left( \lambda \left( 1+t\right) ,\nu \right) 
}{Z\left( \lambda ,\nu \right) }\text{,}  \label{MoMo}
\end{equation}%
and the first two factorial cumulants of the local factorial tilting family $%
\mathrm{FT}(\mu )$ generated by (\ref{COM}) (for given value of $(\lambda
,\nu )$) are 
\begin{equation*}
\mu =\lambda \frac{\dot{Z}\left( \lambda \left( 1+\theta \right) ,\nu
\right) }{Z\left( \lambda \left( 1+\theta \right) ,\nu \right) }\text{ and }%
\lambda ^{2}\frac{\ddot{Z}\left( \lambda \left( 1+\theta \right) ,\nu
\right) }{Z\left( \lambda \left( 1+\theta \right) ,\nu \right) }-\mu ^{2}%
\text{,}
\end{equation*}%
respectively, where dots denote derivatives of $Z\left( \cdot ,\nu \right) $%
. These results confirm, for $\theta =0$, known results for the mean and
variance of the COM-Poisson distribution. It is well known that the
COM-Poisson converges to the Bernoulli distribution $\mathrm{Bi}(\lambda
/\left( 1+\lambda \right) ,1)$ as $\nu \rightarrow \infty $. A similar
result using the above binomial convergence result is obtained by using the
following asymptotic relation: $\log Z\left( \lambda \left( 1+\theta \right)
,\nu \right) \sim \log \left[ 1+\lambda \left( 1+\theta \right) \right] $ as 
$\theta \downarrow -1$, which implies that the local dispersion function of
the factorial tilting family $\mathrm{FT}(\mu )$ defined by (\ref{MoMo})
satisfies $v(\mu )\sim -\mu ^{2}$ as $\mu \downarrow 0$. Using (\ref{BiMo}),
we hence obtain the following Bernoulli convergence:%
\begin{equation*}
c^{-1}\cdot \mathrm{FT}(c\mu )\overset{D}{\longrightarrow }\mathrm{Bi}(\mu
,1)\qquad \text{as }c\downarrow 0.
\end{equation*}%
Turning now to the question of infinite divisibility, we note that \cite%
{Kokonendji2008} argued that the COM-Poisson is a weighted Poisson
distribution of the form (\ref{prob}) with a logconvex (logconcave) weight
function for $0<\nu <1$ ($\nu >1$), and is hence overdispersed
(underdispersed) with respect to the Poisson case $\nu =1$. Following \cite%
{Kokonendji2008}, we may further argue that the distribution is infinitely
divisible in the logconvex case $0<\nu <1$, whereas for $\nu >1$, we have a
discrete underdispersed distribution, which cannot be infinitely divisible.
In order to apply Theorem \ref{converges}, we need the following asymptotic
expansion, gleaned from Sellers \emph{et al.} (2012)%
\nocite{Sellers2012}%
,%
\begin{equation}
C\left( t\right) \sim \alpha ^{-1}\lambda ^{\alpha }\left( 1+t\right)
^{\alpha }-\frac{1-\alpha }{2}\log \lambda \left( 1+t\right) \text{ as }%
t\rightarrow \infty ,  \label{Cit}
\end{equation}%
where $\alpha =1/\nu >0;$ the first term of the expansion being the leading
term. Let $p$ be related to $\alpha $ via (\ref{alfap}). In the
overdispersed case $\alpha >1$ ($p<1$), the result (\ref{Cit}) implies that
the dispersion function is power asymptotic, 
\begin{equation}
v(\mu )\sim \left( \alpha -1\right) \lambda ^{\alpha /\left( \alpha
-1\right) }\mu ^{p}  \label{poder}
\end{equation}%
as $\mu \rightarrow \infty $. The overdispersed factorial dispersion model $%
\mathrm{FD}(\mu ,\gamma )$ generated by (\ref{COM}) hence satisfies (\ref%
{convergence}) as $c\rightarrow \infty $. This result is remarkable, in that
it proves the existence of the power dispersion model $\mathrm{PT}_{p}(\mu
,\gamma )$ in the cases $p<0$ and $0<p<1$ (cf. Theorem \ref{characterize}
and Table \ref{tablep} above), because Theorem \ref{Moraq} implies the
existence of the factorial dispersion model corresponding to the limiting
local dispersion function. In the underdispersed case $0<\alpha <1$ ($p>2$)
we find that (\ref{poder}) is now satisfied as $\mu \downarrow 0$, but with
a negative coefficient for $\mu ^{p}$. In this case, however, Theorem \ref%
{converges} does not apply, because the case $c\downarrow 0$ in (\ref%
{convergence}) requires infinite divisibility, which we do not have in the
underdispersed case, nor do we seem to have infinite dilatability. These
results, while interesting on their own, are to some extent tangential to
the COM-Poisson distribution itself, because the factorial dispersion model $%
\mathrm{FD}(\mu ,\gamma )$ is not contained in the COM-Poisson family.
\end{example}

\subsection{Neyman Type A convergence}

A new result that emerges from Poisson-Tweedie asymptotics is convergence to
the Neyman Type A distribution $\mathrm{PT}_{1}(\mu ,\gamma )$ (cf. Section %
\ref{Pss}). This is the case, in particular, for a certain type of Poisson
mixtures.

Let us assume that the exponential dispersion model $\mathrm{ED}(\mu ,\gamma
)$ has unit variance function satisfying $V(\mu )\sim \mu $ as $\mu
\downarrow 0$ or $\mu \rightarrow \infty $. From Proposition \ref{samevefu}
we obtain that the Poisson mixture $P(\mathrm{ED}(\mu ,\gamma );1)$ is a
factorial dispersion model with unit dispersion function $v=V$. It then
follows from Theorem \ref{converges}, that the corresponding tilted and
dilated model converges to the Neyman Type A distribution, 
\begin{equation*}
c^{-1}\cdot \mathrm{FD}(c\mu ,c\gamma )\overset{D}{\rightarrow }\text{$%
\mathrm{PT}$}_{1}(\mu ,\gamma )\qquad \text{ as }c\downarrow 0\text{ or}%
\rightarrow \infty \text{,}
\end{equation*}%
respectively.

Let $\mathrm{ED}^{\ast }(\mu ,\lambda )$ denote an additive exponential
dispersion model generated by a distribution with an atom at zero, and such
that $(0,1)$ is the largest interval starting at zero with zero probability.
This may happen if the distributions has support $\mathbb{N}_{0}$, but the
distribution need not necessarily be discrete as long as there is an atom at
zero and positive probability at 1 or starting at 1. Then we know from \cite%
{Jorgensen1994} that the unit variance function satisfies $V(\mu )\sim \mu $
as $\mu \downarrow 0$. The corresponding exponential dispersion model is $%
\gamma \mathrm{ED}^{\ast }(\mu ,\gamma ^{-1}).$ Hence, let us consider the
factorial dispersion model $\mathrm{FD}(\mu ,\gamma )$ defined by the
Poisson mixture $P(\gamma \mathrm{ED}^{\ast }(\mu ,\gamma ^{-1});1)$. Then
we have the following large-sample convergence result: 
\begin{equation}
n\cdot \mathtt{\mathrm{FD}}(\mu /n,\gamma /n)\overset{D}{\rightarrow }%
\mathrm{PT}_{1}(\mu ,\gamma )\text{ }\qquad \text{as }n\rightarrow \infty 
\text{.}  \label{NTA}
\end{equation}%
Here, using (\ref{average}), the left-hand side of (\ref{NTA}) may be
interpreted as the sum of $n$ i.i.d. random variables with distribution $%
\mathtt{\mathrm{FD}}(\mu /n,\gamma )$.

\subsection{Hermite convergence revisited}

The Hermite distribution \textrm{$\mathrm{PT}$}$_{0}(\mu ,\gamma )$ of
Example \ref{Lapl} has power dispersion function with $p=0$ ($\alpha =2$),
although it is not a Poisson-Tweedie mixture. In Section \ref{sec:Poisson},
we have already considered a type of Hermite convergence similar to the
central limit theorem. We now consider Hermite convergence based on Theorem %
\ref{converges}.

It is important to keep in mind that the parameters of the Hermite
distribution must satisfy $0<\gamma \leq \mu $. The dilation property (\ref%
{dilation}) hence takes the following form: 
\begin{equation*}
c\cdot \mathrm{PT}_{0}(\mu ,\gamma )=\mathrm{PT}_{0}(c\mu ,c^{2}\gamma )%
\text{ for }c<1\text{,}
\end{equation*}%
where the restriction $c<1$ ensures that the transformed parameters satisfy
the condition. As a consequence, only the case $c\downarrow 0$ of (\ref%
{convergence}) is available in the Hermite case.

We first note that the local dispersion function of any locally
overdispersed factorial dispersion model $\mathrm{FD}(\mu ,\gamma )$ with $%
0\in \mathrm{int}\Psi _{0}$ satisfies $v(\mu )\sim v(0)>0$ as $\mu
\downarrow 0$. By Theorem \ref{converges}, and using the form (\ref{larsam})
we obtain convergence to the Hermite distribution,%
\begin{equation}
n^{1/2}\cdot \mathrm{FD}(n^{-1/2}\mu ,\gamma /n)\overset{D}{\rightarrow }%
\mathrm{PT}_{0}(\mu ,\gamma v(0))\text{ as }n\rightarrow \infty \text{,}
\label{none}
\end{equation}%
for each $\mu >0$, provided that $\gamma v(0)\leq \mu $. By referring once
more to Eq. (\ref{average}), we note that the left-hand side of (\ref{none})
involves a dilated and factorially tilted dilation average of $n$ i.i.d.
variables from the distribution $\mathrm{FD}(n^{-1/2}\mu ,\gamma )$.

The condition that $\gamma v(0)\leq \mu $ may be alleviated by means of
Poisson translation, similar to the procedure of Section \ref{sec:Poisson}.
This leads to the following result: 
\begin{equation*}
n^{1/2}\cdot \left[ \mathrm{FD}(\mu _{0}+n^{-1/2}\mu ,\gamma /n)\ominus \mu
_{0}\right] \overset{D}{\rightarrow }\mathrm{PT}_{0}(\mu ,\gamma v(\mu _{0}))%
\text{ as }n\rightarrow \infty \text{,}
\end{equation*}%
for $\mu _{0}\in \Psi _{0}$, provided that $\mu $ is chosen such that $%
\gamma v(\mu _{0})\leq \mu $. Here $\ominus $ denotes Poisson subtraction,
as defined by (\ref{substral}).

\section{Multivariate discrete dispersion models\label{sec:Multivariate}}

We now consider multivariate generalizations of some of the above results,
in particular a multivariate Poisson-Tweedie model (cf. Section \ref%
{sec:multPT}). We refer to \cite{Johnson1997} for general results on
multivariate discrete distributions.

\subsection{Multivariate factorial cumulants and other properties\label%
{sec:fcgf}}

If $\boldsymbol{X}$ is a $k$-variate random vector, and $\boldsymbol{s}$ a $%
k $-vector with non-negative elements, we use the notation $\boldsymbol{s}^{%
\boldsymbol{X}}=s_{1}^{X_{1}}\cdots s_{k}^{X_{k}}$. The multivariate FCGF 
\cite[p.~4]{Johnson1997} is defined by 
\begin{equation*}
C(\boldsymbol{t};\boldsymbol{X})=\log \mathrm{E}\left[ \left( \boldsymbol{1}+%
\boldsymbol{t}\right) ^{\boldsymbol{X}}\right] \text{ for }\boldsymbol{t}%
\geq -\boldsymbol{1}\text{,}
\end{equation*}%
where $\boldsymbol{1}$ is a vector of ones, and the inequality $\boldsymbol{t%
}\geq -\boldsymbol{1}$ is understood elementwise$.$ The effective domain for 
$C$ is defined by $\mathrm{dom}(C)=\left\{ \boldsymbol{t}\geq -\boldsymbol{1}%
:C(\boldsymbol{t})<\infty \right\} $. When $\boldsymbol{0}\in \mathrm{int}(%
\mathrm{dom}(C))$, the mean vector is $\mathrm{E}\left( \boldsymbol{X}%
\right) =\dot{C}(\boldsymbol{0};\boldsymbol{X})$, and the \emph{dispersion
matrix} $\mathrm{S}(\boldsymbol{X})=\ddot{C}(\mathbf{0};\boldsymbol{X})=%
\mathrm{Cov}\left( \boldsymbol{X}\right) -\mathrm{diag}\left[ \mathrm{E}%
\left( \boldsymbol{X}\right) \right] $ is a $k\times k$ symmetric matrix
with entries%
\begin{equation*}
\mathrm{S}_{ij}(\boldsymbol{X})=\left\{ 
\begin{array}{ccc}
\mathrm{S}(X_{i}) & \text{ for } & i=j \\ 
\mathrm{Cov}(X_{i},X_{j}) & \text{ for } & i\neq j\text{.}%
\end{array}%
\right.
\end{equation*}

We now present a new definition of multivariate over/underdispersion based
on the dispersion matrix. We say that the random vector $\boldsymbol{X}$ is 
\emph{equidispersed} if $\mathrm{S}(\boldsymbol{X})=0$. If $\boldsymbol{X}$
is not equidispersed, it is called \emph{over/underdispersed} if the
dispersion matrix $\mathrm{S}(\boldsymbol{X})$ is positive/negative
semidefinite, i.e. $\mathrm{S}(\boldsymbol{X})$ has at least one
positive/negative eigenvalue, respectively. We say that the dispersion of $%
\boldsymbol{X}$ is \emph{indefinite} if $\mathrm{S}(\boldsymbol{X})$ has
both positive and negative eigenvalues.

As an example, consider the bivariate Poisson distribution defined by 
\begin{equation}
\left[ 
\begin{array}{c}
X_{1} \\ 
X_{2}%
\end{array}%
\right] =\left[ 
\begin{array}{c}
U_{1}+U_{2} \\ 
U_{1}+U_{3}%
\end{array}%
\right] ,  \label{dobpoi}
\end{equation}%
where $U_{i}\sim \mathrm{Po}(\mu _{i}),$ $i=1,2,3$ are independent Poisson
random variables. The two marginals $X_{1}$ and $X_{2}$ are equidispersed,
and provided that $\mu _{1}>0$, the marginals are positively correlated, in
which case the dispersion is indefinite. In the independence case $\mu
_{1}=0 $ we find that $\boldsymbol{X}$ is equidispersed. More generally, if
the marginals of $\boldsymbol{X}$ are independent and Poisson distributed
with mean vector $\boldsymbol{\mu }\geq \boldsymbol{0}$, we obtain the FCGF 
\begin{equation}
C(\boldsymbol{t};\boldsymbol{X})=\boldsymbol{\mu }^{\top }\boldsymbol{t}%
\text{,}  \label{Poim}
\end{equation}%
which is linear, and hence equidispersed. We note in passing, that the
multivariate Poisson FCGF (\ref{Poim}) is of homogeneous type, i.e. of the
form $K\left( \boldsymbol{t}^{\top }\boldsymbol{\mu }\right) $, where $%
K\left( 0\right) =0$, see 
\citet[p.~19]{Johnson1997}%
. The multinomial distribution $\boldsymbol{X}\sim \mathrm{Mu}(\boldsymbol{%
\mu },n)$ has dispersion matrix $\mathrm{S}(\boldsymbol{X})=-n\boldsymbol{%
\mu \mu }^{\top }$, making this distribution underdispersed.

We now derive the scaling properties of the dispersion matrix $\mathrm{S}%
\left( \boldsymbol{X}\right) $ with respect to dilation, generalizing the
results of Section \ref{sec:Dil}. For a random vector $\boldsymbol{X}$, we
define the\emph{\ dilation} \emph{linear combination} $\boldsymbol{c\cdot X}$
with coefficient vector $\boldsymbol{c}\geq \boldsymbol{0}$ ($1\times k$) as
follows: 
\begin{equation*}
C(t;\boldsymbol{c\cdot X})=C(\boldsymbol{c}^{\top }t;\boldsymbol{X})
\end{equation*}%
provided that the right-hand side is a (univariate) FCGF. The mean and
dispersion matrix of a dilation linear combination are given by 
\begin{equation}
\mathrm{E}\left( \boldsymbol{c\cdot X}\right) =\boldsymbol{c}\mathrm{E}(%
\boldsymbol{X})\text{ and }\mathrm{S}\left( \boldsymbol{c\cdot X}\right) =%
\boldsymbol{c}\mathrm{S}(\boldsymbol{X})\boldsymbol{c}^{\top }\text{,}
\label{produ}
\end{equation}%
respectively. It follows that if $\boldsymbol{c\cdot X}$ is equidispersed
for some $\boldsymbol{c}\neq \boldsymbol{0}$, then the dispersion matrix $%
\mathrm{S}(\boldsymbol{X})$ is singular. The reverse implication holds if
the vector $\boldsymbol{c}\geq \boldsymbol{0}$ is such that $\boldsymbol{c}%
\mathrm{S}(\boldsymbol{X})\boldsymbol{c}^{\top }=0$. Similarly, for an $\ell
\times k$ matrix $\boldsymbol{A}\geq \boldsymbol{0}$ we define $\boldsymbol{%
A\cdot X}$ by 
\begin{equation*}
C(\boldsymbol{t};\boldsymbol{A\cdot X})=C(\boldsymbol{A}^{\top }\boldsymbol{t%
};\boldsymbol{X})\text{,}
\end{equation*}%
again provided that the right-hand side is an FCGF. For a multivariate
Poison random vector $\boldsymbol{X}$ with FCGF (\ref{Poim}) this yields the
following transformation 
\begin{equation*}
C(\boldsymbol{t};\boldsymbol{A\cdot X})=\boldsymbol{\mu }^{\top }\boldsymbol{%
A}^{\top }\boldsymbol{t}=\left( \boldsymbol{A\mu }\right) ^{\top }%
\boldsymbol{t}
\end{equation*}%
making $\boldsymbol{A\cdot X}$ multivariate Poisson with mean $\boldsymbol{%
A\mu }$.

We now turn to a multivariate version of the law of thin numbers. We define
the dilation average for the i.i.d. sequence $\boldsymbol{X}_{1},\boldsymbol{%
X}_{2},\ldots $ of $k\times 1$ random vectors by 
\begin{equation*}
\overline{\boldsymbol{X}}_{n}=\left( n\boldsymbol{I}\right) ^{-1}\cdot 
\boldsymbol{S}_{n}\text{.}
\end{equation*}%
where $\boldsymbol{S}_{n}=\boldsymbol{X}_{1}+\cdots +\boldsymbol{X}_{n}$
denotes the $n$th partial sum and $\boldsymbol{I}$ is the identity matrix.
We assume that the $\boldsymbol{X}_{i}$ are discrete with mean vector $%
\boldsymbol{\mu \geq 0}$. Similar to the univariate case in Section \ref%
{sec:Poisson}, we obtain the FCGF for $\overline{\boldsymbol{X}}_{n}$ as
follows: 
\begin{eqnarray}
C\left( \boldsymbol{t};\overline{\boldsymbol{X}}_{n}\right) &=&C\left(
\left( n\boldsymbol{I}\right) ^{-1}\boldsymbol{t};\boldsymbol{X}_{1}+\cdots +%
\boldsymbol{X}_{n}\right)  \notag \\
&=&nC(n^{-1}\boldsymbol{t};\boldsymbol{X}_{1})=\boldsymbol{\mu }^{\top }%
\boldsymbol{t}+O(n^{-1}),  \label{cconv}
\end{eqnarray}%
which converges to the multivariate Poisson FCGF (\ref{Poim}) as $%
n\rightarrow \infty $.

To show Hermite convergence, we consider an i.i.d. sequence of discrete
random vectors $\boldsymbol{X}_{i}$ with $\mathrm{E}(\boldsymbol{X}_{1})=%
\boldsymbol{m}>\boldsymbol{0}$ and $\mathrm{S}(\boldsymbol{X}_{1})=%
\boldsymbol{\Sigma }>\boldsymbol{0}$. Define the translated standardized
variable $\boldsymbol{Z}_{n}\left( \boldsymbol{\mu }\right) $ for $%
\boldsymbol{\mu }\geq \boldsymbol{\Sigma 1}$ by%
\begin{equation*}
\boldsymbol{Z}_{n}\left( \boldsymbol{\mu }\right) =\left( n\boldsymbol{I}%
\right) ^{-1/2}\cdot \left[ \boldsymbol{S}_{n}\ominus \left( n\boldsymbol{m}%
-n^{1/2}\boldsymbol{\mu }\right) \right] ,
\end{equation*}%
where the Poisson subtraction $\ominus $ is defined by analogy with (\ref%
{substral}). By expanding the FCGF of $\boldsymbol{Z}_{n}\left( \boldsymbol{%
\mu }\right) $ we obtain 
\begin{equation*}
C(\boldsymbol{t};\boldsymbol{Z}_{n}\left( \boldsymbol{\mu }\right) )=%
\boldsymbol{\mu }^{\top }\boldsymbol{t}+\frac{1}{2}\boldsymbol{t}^{\top }%
\boldsymbol{\Sigma t}+O\left( n^{-1/2}\right) \text{,}
\end{equation*}%
which shows that the multivariate Hermite distribution of 
\citep[cf.][p.~274]{Johnson1997}
with mean vector $\boldsymbol{\mu }$ and dispersion matrix $\boldsymbol{%
\Sigma }$ appears in the limit as $n\rightarrow \infty $. Hence $\boldsymbol{%
Z}_{n}\left( \boldsymbol{\mu }\right) $ converges in distribution to the
multivariate Hermite distribution.

Finally, let us consider a multivariate generalization of the zero-inflation
index (\ref{ZI}), namely 
\begin{equation}
\mathrm{ZI}\left( \boldsymbol{X}\right) =1+\frac{\log P(\boldsymbol{X}=%
\boldsymbol{0})}{\mathrm{E}\left( X_{1}\right) +\cdots +\mathrm{E}\left(
X_{k}\right) }=1+\frac{C(-\boldsymbol{1};\boldsymbol{X})}{\boldsymbol{1}%
^{\top }\dot{C}(\boldsymbol{0};\boldsymbol{X})}\text{.}  \label{ZIM}
\end{equation}%
This index measures zero-inflation/deflation relative to independent Poisson
random variables (equidispersion) with the same total mean, corresponding to
positive/negative values of $\mathrm{ZI}\left( \boldsymbol{X}\right) $,
respectively. It is useful to extend this to a directional measure of
zero-inflation, namely 
\begin{equation*}
\mathrm{ZI}\left( \boldsymbol{c\cdot X}\right) =1+\frac{C(-1;\boldsymbol{%
c\cdot X})}{\dot{C}(\boldsymbol{0};\boldsymbol{c\cdot X})}=1+\frac{C(-%
\boldsymbol{c}^{\top };\boldsymbol{X})}{\boldsymbol{c}\dot{C}(\boldsymbol{0};%
\boldsymbol{X})}
\end{equation*}%
which reduces to (\ref{ZIM}) for $\boldsymbol{c=1}^{\top }$. This index
measures zero-inflation/deflation for the dilation linear combination $%
\boldsymbol{c\cdot X}$ as a function of $\boldsymbol{c}$.

\subsection{Multivariate Poisson-Tweedie models\label{sec:multPT}}

We now introduce a new class of multivariate Poisson-Tweedie mixtures, which
is based on the multivariate Tweedie distributions of \cite{Jorgensen2012}.
Consider the $k$-variate Tweedie distribution $\boldsymbol{Y}\sim \mathrm{Tw}%
_{p}(\boldsymbol{\mu },\boldsymbol{\Sigma )}$ with mean vector $\boldsymbol{%
\mu }$ and covariance matrix 
\begin{equation}
\mathrm{Cov}(\boldsymbol{Y})=\left[ \boldsymbol{\mu }\right] ^{p/2}%
\boldsymbol{\Sigma }\left[ \boldsymbol{\mu }\right] ^{p/2}  \label{CovY}
\end{equation}%
where $\boldsymbol{\Sigma }$ denotes a $k\times k$ symmetric
positive-definite matrix, and the notation $\left[ \boldsymbol{\mu }\right]
^{p/2}$ denotes a power of the diagonal matrix $\left[ \boldsymbol{\mu }%
\right] =\mathrm{diag}(\boldsymbol{\mu })$. By construction, this
distribution has univariate Tweedie marginals, see \cite{Jorgensen2012}.

Let us define the \emph{multivariate Poisson-Tweedie model }$\boldsymbol{%
X\sim }\mathrm{PT}_{p}(\boldsymbol{\mu },\boldsymbol{\Sigma )}$ as a Poisson
mixture 
\begin{equation*}
\boldsymbol{X}|\boldsymbol{Y\sim }\text{ independent }\mathrm{Po}(Y_{i})%
\text{ for }i=1,\ldots ,k\text{,}
\end{equation*}%
where $X_{1},\ldots ,X_{k}$ are assumed conditionally independent given $%
\boldsymbol{Y}$. The multivariate Poisson-Tweedie model has univariate
Poisson-Tweedie margins, $X_{i}\sim \mathrm{PT}_{p}(\mu _{i},\sigma _{ii}),$
where $\sigma _{ij}$ denote the entries of $\boldsymbol{\Sigma }$. The mean
vector is $\boldsymbol{\mu }$ and the dispersion matrix is (\ref{CovY})
(positive-definite) making the distribution overdispersed. The covariance
matrix for $\boldsymbol{X}$ has the form 
\begin{equation*}
\mathrm{Cov}(\boldsymbol{X})=\left[ \boldsymbol{\mu }\right] +\left[ 
\boldsymbol{\mu }\right] ^{p/2}\boldsymbol{\Sigma }\left[ \boldsymbol{\mu }%
\right] ^{p/2}\text{,}
\end{equation*}%
making it straightforward to fit multivariate Poisson-Tweedie regression
models using quasi-likelihood. The multivariate Poisson-Tweedie model
satisfies the following dilation property:%
\begin{equation*}
\left[ \boldsymbol{c}\right] \cdot \mathrm{PT}_{p}(\boldsymbol{\mu },%
\boldsymbol{\Sigma )}=\mathrm{PT}_{p}(\left[ \boldsymbol{c}\right] 
\boldsymbol{\mu },\left[ \boldsymbol{c}\right] ^{1-p/2}\boldsymbol{\Sigma }%
\left[ \boldsymbol{c}\right] ^{1-p/2}),
\end{equation*}%
where $\boldsymbol{c}$ is a $k$-vector with positive elements, generalizing
the univariate dilation property (\ref{dilation}). In this way, we obtain
multivariate generalizations of all the Poisson-Tweedie models of Table \ref%
{tablep} for $p\geq 1$, including multivariate Neyman Type A, Pólya-Aeppli,
negative binomial and Poisson-inverse Gaussian distributions.

\section{Discussion\label{sec7}}

In this paper have developed a new class of discrete factorial dispersion
models based on exploring the properties of the factorial cumulant
generating function, and we have shown that the dispersion function is a
powerful characterization and convergence tool for factorial dispersion
models. In particular, the Poisson-Tweedie convergence theorem implies that
Poisson-Tweedie models are likely to appear frequently in practice, making
these models especially useful for modelling overdispersed count data. These
results depend in a crucial way on interpreting the dilation operator as a
discrete analogue of scaling.

These results show that factorial dispersion models are in many ways
analogous to exponential dispersion models and to the recently proposed
classes of extreme and geometric dispersion models 
\citep{Jorgensen2007a,Jorgensen2010}%
. A common trait for these four types of dispersion models is the role of
power asymptotics, which in the extreme dispersion model case implies some
of the classical convergence results for extremes towards generalized
extreme value distributions (Weibull, Fréchet and Gumbel distributions), see 
\citet{Jorgensen2007a}
for details. It seems likely that there exist further types of dispersion
models with a similar structure, for example in free probability, where \cite%
{Bryc2008} has introduced so-called free exponential families, and studied
an analogue of quadratic variance functions.

Many of our results have multivariate analogues, and in particular we have
introduced a class of multivariate Poisson-Tweedie mixtures with
Poisson-Tweedie margins. We have introduced a multivariate notion of over-
and underdispersion, and a multivariate zero-inflation index. We have also
shown that the dilation properties of the dispersion matrix are similar to
the scaling properties of the covariance matrix.

There remain a number of further questions to be dealt with for factorial
dispersion models. In particular, we need to develop methods for probability
calculations and simulations further. We are currently developing methods
for quasi-likelihood estimation and inference for multivariate
Poisson-Tweedie models, along the same lines as \cite{Jorgensen2011a}. We
would also like to obtain a better understanding of underdispersion for
factorial dispersion models, perhaps based on the M-transformation, where,
however, we are faced with the problem of deciding on the existence of the
M-transformation in each case. Finally, it seems possible to obtain new
types of point processes based on infinitely divisible factorial dispersion
models. In particular, point processes based on Poisson-Tweedie models would
seem to have useful dilation properties.

\section*{Acknowledgements}

We are grateful to Christian Weiß for useful comments on a previous version
of the paper.

\section*{Appendix A: Exponential dispersion models}

In this appendix, we summarize some relevant facts about exponential
dispersion models and Tweedie models. An exponential dispersion model $%
\mathrm{ED}(\mu ,\gamma )$ with mean $\mu \in \Omega $, dispersion parameter 
$\gamma >0$ and unit variance function $V(\mu )$ has PDF of the form%
\begin{equation}
f(y;\mu ,\gamma )=a(y;\gamma )\exp \left[ -\frac{1}{2\gamma }d(y;\mu )\right]
\text{ for }y,\mu \in \Omega \text{,}  \label{ED}
\end{equation}%
where the unit deviance function $d(y;\mu )$ is defined by 
\begin{equation*}
d(y;\mu )=2\int_{\mu }^{y}\frac{y-z}{V(z)}\,dz\text{ for }y,\mu \in \Omega 
\text{.}
\end{equation*}%
The model (\ref{ED}) is, for each known value of $\gamma $, a natural
exponential family with variance function $\gamma V(\mu )$. Hence, the
function $a(y;\gamma )$ may be determined by Fourier inversion from the CGF,
which may in turn be obtained from $V$. The model $\mathrm{ED}(\mu ,\gamma )$
satisfies the following reproductive property:%
\begin{equation}
\overline{Y}_{n}\sim \mathrm{ED}(\mu ,\gamma /n),  \label{ybar}
\end{equation}%
where $\overline{Y}_{n}$ is the average of $Y_{1},\ldots ,Y_{n}$, which are
i.i.d. from $\mathrm{ED}(\mu ,\gamma )$.

The Tweedie exponential dispersion model $\mathrm{Tw}_{p}(\mu ,\gamma )$ has
mean $\mu $ and unit variance function%
\begin{equation*}
V(\mu )=\mu ^{p}\qquad \text{for }\quad \mu \in \Omega _{p}\text{,}\qquad 
\text{where }\quad p\notin \left( 0,1\right) .
\end{equation*}%
The domain for $\mu $ is either $\Omega _{0}=R$ or $\Omega _{p}=R_{+}$ for $%
p\neq 0$. Tweedie models satisfy the scaling property 
\begin{equation}
c\text{$\mathrm{Tw}$}_{p}(\mu ,\gamma )=\mathtt{\mathrm{Tw}}_{p}(c\mu
,c^{2-p}\gamma )\text{ }\qquad \text{for }c>0\text{.}  \label{scaling}
\end{equation}

Conventional Tweedie asymptotics 
\citep{Jorgensen1994}
have the following form. If $\mathtt{\mathrm{ED}}(\mu ,\gamma )$ with unit
variance function $V(\mu )$ satisfies%
\begin{equation*}
V(\mu )\sim \mu ^{p}\text{ as }\mu \downarrow 0\text{ or }\mu \rightarrow
\infty
\end{equation*}%
then 
\begin{equation}
c^{-1}\mathtt{\mathrm{ED}}(c\mu ,c^{2-p}\gamma )\overset{D}{\rightarrow }%
\text{$\mathrm{Tw}$}_{p}(\mu ,\gamma )\text{ }\qquad \text{as }c\downarrow 0%
\text{ or }c\rightarrow \infty \,\text{,}  \label{conv}
\end{equation}%
respectively. The proof is based on convergence of the variance function on
the left-hand side of (\ref{conv}), 
\begin{equation*}
c^{-2}c^{2-p}\gamma V(c\mu )\rightarrow \gamma \mu ^{p}\text{,}
\end{equation*}%
applying \nocite{Mora1990}Mora's (1990) convergence theorem. The case $%
c^{2-p}\rightarrow \infty $ requires the model $\mathtt{\mathrm{ED}}(\mu
,\gamma )$ to be infinite divisible. This result implies a Tweedie
approximation, by means of (\ref{scaling}) 
\begin{equation*}
\mathtt{\mathrm{ED}}(c\mu ,c^{2-p}\gamma )\overset{\cdot }{\sim }\mathtt{%
\mathrm{Tw}}_{p}(c\mu ,c^{2-p}\gamma )\text{ }\qquad \text{as }c\downarrow 0%
\text{ or }c\rightarrow \infty \text{.}
\end{equation*}

In some cases, we have a large-sample interpretation of Tweedie convergence.
Let us consider the average\textbf{\ }$\bar{Y}_{n}$ with distribution (\ref%
{ybar}). Then for $p\neq 2$ we obtain 
\begin{equation*}
n^{-1/(p-2)}\mathtt{\mathrm{ED}}(n^{1/(p-2)}\mu ,\gamma /n)\overset{D}{%
\rightarrow }\text{$\mathrm{Tw}$}_{p}(\mu ,\gamma )\text{ }\qquad \text{as }%
n\rightarrow \infty \text{.}
\end{equation*}%
We interpret this result as saying that the scaled and exponentially tilted
average $\bar{Y}_{n}$ converges to a Tweedie distribution as $n\rightarrow
\infty $.

\section*{Appendix B: Proof of Theorem 4.1\label{secApp}}

Consider a sequence of factorial tilting families $\mathrm{FT}_{n}(\mu )$
with local dispersion functions $v_{n}$ having domain $\Psi _{n}$ and FCGF $%
C_{n}$ satisfying the conditions of Theorem \ref{Moraq}. The idea of the
proof is to obtain the FCGF derivative $\dot{C}$ from the limiting
dispersion function $v$, and in turn use the uniform convergence to show
convergence of the sequence $C_{n}$.

We begin by considering the nonzero case, where $v(\mu )\neq 0$ for $\mu \in
\Psi _{0}$. Let $K$ be a given compact subinterval of $\Psi _{0}$. By
assumption $\Psi _{0}\subseteq \mathrm{int}\left( \lim \Psi _{n}\right) $,
so we may assume that $K\subseteq \Psi _{n}$ from some $n_{0}$ on. We only
need to consider $n>n_{0}$ from now on. Fix a $\mu _{0}\in \mathrm{int}\,K$.
Let $\psi _{n}=\dot{C}_{n}^{-1}$ denote the inverse FCGF derivative defined
by $\dot{\psi}_{n}\left( \mu \right) =1/v_{n}(\mu )$ on $\Psi _{n}$ and $%
\psi _{n}\left( \mu _{0}\right) =0$. Let $\dot{C}_{n}$, $C_{n}$ etc. denote
the quantities associated with this parametrization. Similarly, define $\psi
:\Psi _{0}\rightarrow \mathbb{R}$ by $\dot{\psi}\left( \mu \right) =1/v(\mu
) $ on $\Psi _{0}$ and $\psi (\mu _{0})=0$. Then for $\mu \in K$ 
\begin{equation}
\left\vert \dot{\psi}_{n}\left( \mu \right) -\dot{\psi}\left( \mu \right)
\right\vert =\frac{\left\vert v_{n}(\mu )-v(\mu )\right\vert }{v_{n}(\mu
)v(\mu )}\text{.}  \label{psin}
\end{equation}%
By the uniform convergence of $v_{n}(\mu )$ to $v(\mu )$ on $K$, it follows
that $\left\{ v_{n}(\mu )\right\} $ is uniformly bounded on $K$. Since $%
v(\mu )$ is bounded on $K$, it follows from (\ref{psin}) and from the
uniform convergence of $v_{n}$ that $\dot{\psi}_{n}\left( \mu \right)
\rightarrow \dot{\psi}\left( \mu \right) $ uniformly on $K$. This and the
fact that $\psi _{n}\left( \mu _{0}\right) =\psi (\mu _{0})$ for all $n$
implies, by a result from \nocite{Rudin1976}Rudin (1976, Theorem 7.17), that 
$\psi _{n}\left( \mu \right) \rightarrow \psi \left( \mu \right) $ uniformly
on $K$. Since $K$ was arbitrary, we have $\psi _{n}\left( \mu \right)
\rightarrow \psi \left( \mu \right) $ for all $\mu \in \Psi _{0}$.

Let $I_{n}=\psi _{n}\left( \Psi _{n}\right) $ and $I_{0}=\psi (\Psi
_{0})\subseteq \mathrm{int}\left( \lim I_{n}\right) $. Let $J=\psi
(K)\subseteq I_{0}$ and $J_{n}=\psi _{n}(K)\subseteq I_{n}$. Define $\dot{C}%
:I_{0}\rightarrow \Psi _{0}$ by $\dot{C}(y)=\psi ^{-1}(y)$. Since $\psi $ is
strictly monotone and differentiable, the same is the case for $\dot{C}$.
Let $\mu \in K$ be given and let $y=\psi (\mu )\in J$ and $y_{n}=\psi
_{n}(\mu )\in J_{n}$. Since $v_{n}(\mu )$ is uniformly bounded on $K$, there
exists an $M>0$ such that $\left\vert v_{n}(\mu )\right\vert \leq M$ for all 
$n$ and $\mu \in K$. It follows that $\left\vert \ddot{C}_{n}(y)\right\vert
=\left\vert v_{n}\left( \dot{C}_{n}(y)\right) \right\vert \leq M$ for all $%
y\in J$ due to the fact that $J\subseteq J_{n}$ for $n$ large enough. Since $%
\mu =\dot{C}(y)=\dot{C}_{n}(y_{n})$ we find, using the mean value theorem,
that 
\begin{eqnarray*}
\left\vert \dot{C}_{n}(y)-\dot{C}(y)\right\vert &=&\left\vert \dot{C}_{n}(y)-%
\dot{C}_{n}(y_{n})\right\vert \\
\ &\leq &M\left\vert y-y_{n}\right\vert \\
\ &=&M\left\vert \psi (\mu )-\psi _{n}(\mu )\right\vert \text{.}
\end{eqnarray*}%
This implies that $\dot{C}_{n}(y)\rightarrow \dot{C}(y)$ uniformly in $y\in
J $. Since $C(0)=C_{n}(0)$ for all $n$, it follows by similar arguments as
above that $C_{n}(y)\rightarrow C(y)$ uniformly on $J$. We conclude from the
convergence of the sequence of MGFs $\exp \left[ C_{n}\left( e^{s}-1\right) %
\right] \rightarrow \exp \left[ C\left( e^{s}-1\right) \right] $ for $s\in
\log \left( J+1\right) $ that the sequence of distributions $\mathrm{FT}%
_{n}(\mu _{0})$ converges weakly to a probability measure $P$ with FCGF $C$.
We let $\mathrm{FT}(\mu )$ denote the factorial tilting family generated by $%
P$ with local dispersion function $v$ on $\Psi _{0}$. We may now complete
the proof in the nonzero case by proceeding like in the proof of Proposition %
\ref{LTN}.

In the case where $v(\mu )=0$ (the zero case), we cannot define the function 
$\psi $ as above. Instead we take $C(t)=t\mu _{0},$ such that $\dot{C}%
(t)=\mu _{0}$ and $\ddot{C}(t)=0$ for $t\in \mathbb{R}$. For any $\epsilon
>0 $, we may choose an $n_{0}$ such that $\left\vert v_{n}(\mu )\right\vert
\leq \epsilon $ for any $n\geq n_{0}$ and $\mu \in K$. For such $n$ and $\mu 
$ we hence obtain%
\begin{equation*}
\left\vert \psi _{n}(\mu )\right\vert =\int_{\mu _{0}}^{\mu }\frac{1}{%
\left\vert v_{n}(t)\right\vert }\,dt\geq \frac{\left\vert \mu -\mu
_{0}\right\vert }{\epsilon }\text{,}
\end{equation*}%
which can be made arbitrarily large by choosing $\epsilon $ small. We hence
conclude that $J_{n}=\psi _{n}(K)\rightarrow \mathbb{R}$ as $n\rightarrow
\infty $.

Now we let $J$ be a compact interval such that $0\in \mathrm{int}J$,
implying that $J\subseteq $ $J_{n}$ for $n$ large enough. For such $n$ we
hence obtain that $\left\vert \ddot{C}_{n}(t)\right\vert =\left\vert v_{n}(%
\dot{C}_{n}(t))\right\vert \leq \epsilon $ for all $t\in J$, because then $%
\dot{C}_{n}(t)\in K$. Since $\mu _{0}=\dot{C}(t)=\dot{C}_{n}(0)$ we find,
again by the mean value theorem, that for $t\in J$, 
\begin{equation*}
\left\vert \dot{C}_{n}(t)-\dot{C}(t)\right\vert =\left\vert \dot{C}_{n}(t)-%
\dot{C}_{n}(0)\right\vert \leq \epsilon \left\vert t\right\vert \text{.}
\end{equation*}%
This implies that $\dot{C}_{n}(t)\rightarrow \dot{C}(t)$ uniformly in $t\in
J $. By similar arguments as above, we conclude that $\mathrm{FT}_{n}(\mu
_{0}) $ converges weakly to a probability measure $P$ with FCGF $C(t)=t\mu
_{0}$, which implies the desired conclusion in the zero case, completing the
proof.

\bibliographystyle{authordate3}
\bibliography{bentj}

\begin{thebibliography}{}

\bibitem[\protect\citename{Barreto-Souza \& Bourguignon,
  }2014]{Barreto-Souza2014}
{\sc Barreto-Souza, W., \& Bourguignon, M.} 2014.
\newblock A skew {INAR(1)} process on {$\mathbb{Z}$}.
\newblock {\em AStA Advances in Statistical Analysis}, {\bf DOI},
  10.1007/s10182--014--0236--2.

\bibitem[\protect\citename{Bryc, }2009]{Bryc2008}
{\sc Bryc, W.} 2009.
\newblock Free exponential families as kernel families.
\newblock {\em Demonstr. Math.}, {\bf XLII}, 657--672.

\bibitem[\protect\citename{Dobbie \& Welsh, }2001]{Dobbie2001}
{\sc Dobbie, M.~J, \& Welsh, A.~H}. 2001.
\newblock Models for zero-inflated count data using the {N}eyman type {A}
  distribution.
\newblock {\em Statistical Modelling}, {\bf 1}, 65--80.

\bibitem[\protect\citename{El-Shaarawi {\em et~al.\ }\relax,
  }2011]{El-Shaarawi2011}
{\sc El-Shaarawi, A.~H., Zhu, R., \& Joe, H.} 2011.
\newblock Modelling species abundance using the {P}oisson-{T}weedie family.
\newblock {\em Environmetrics}, {\bf 22}, 152--164.

\bibitem[\protect\citename{Giles, }2010]{Giles2010}
{\sc Giles, D.~E.} 2010.
\newblock Hermite regression analysis of multi-modal count data.
\newblock {\em Economics Bulletin}, {\bf 30}, 2936--2945.

\bibitem[\protect\citename{Harremoës {\em et~al.\ }\relax,
  }2010]{Harremoes2010}
{\sc Harremoës, P., Johnson, O., \& Kontoyiannis, I.} 2010.
\newblock Thinning, entropy, and the law of thin numbers.
\newblock {\em IEEE Transactions on Information Theory}, {\bf 56}, 4228--4244.

\bibitem[\protect\citename{Jensen \& Nielsen, }1997]{Jensen1997}
{\sc Jensen, S.~T., \& Nielsen, B.} 1997.
\newblock On convergence of multivariate {L}aplace transforms.
\newblock {\em Statist. Probab. Lett.}, {\bf 33}, 125--128.

\bibitem[\protect\citename{Johnson {\em et~al.\ }\relax, }1997]{Johnson1997}
{\sc Johnson, N.~L., Kotz, S., \& Balakrishnan, N.} 1997.
\newblock {\em Discrete Multivariate Distributions}.
\newblock New York: Wiley.

\bibitem[\protect\citename{Johnson {\em et~al.\ }\relax, }2005]{Johnson2005}
{\sc Johnson, N.~L., Kemp, A.~W., \& Kotz, S.} 2005.
\newblock {\em Univariate Discrete Distributions}. 3rd edn.
\newblock Hoboken, N.J.: Wiley.

\bibitem[\protect\citename{J{\o}rgensen, }1997]{Jorgensen1997}
{\sc J{\o}rgensen, B.} 1997.
\newblock {\em The Theory of Dispersion Models}.
\newblock London: Chapman \& Hall.

\bibitem[\protect\citename{J{\o}rgensen \& Kokonendji, }2011]{Jorgensen2010}
{\sc J{\o}rgensen, B., \& Kokonendji, C.~C.} 2011.
\newblock Dispersion models for geometric sums.
\newblock {\em Brazilian J. Probab. Statist.}, {\bf 25}, 263--293.

\bibitem[\protect\citename{J{\o}rgensen \& Martínez, }2013]{Jorgensen2012}
{\sc J{\o}rgensen, B., \& Martínez, J.~R.} 2013.
\newblock Multivariate exponential dispersion models.
\newblock {\em Pages  73--98 of:} {\sc Kollo, T.} (ed), {\em Multivariate
  Statistics: Theory and Applications. Proceedings of the IX Tartu Conference
  on Multivariate Statistics \& XX International Workshop on Matrices and
  Statistics.}
\newblock Singapore: World Scientific.

\bibitem[\protect\citename{J{\o}rgensen {\em et~al.\ }\relax,
  }1994]{Jorgensen1994}
{\sc J{\o}rgensen, B., Martínez, J.~R., \& Tsao, M.} 1994.
\newblock Asymptotic behaviour of the variance function.
\newblock {\em Scand. J. Statist.}, {\bf 21}, 223--243.

\bibitem[\protect\citename{J{\o}rgensen {\em et~al.\ }\relax,
  }2010]{Jorgensen2007a}
{\sc J{\o}rgensen, B., Goegebeur, Y., \& Martínez, J.~R.} 2010.
\newblock Dispersion models for extremes.
\newblock {\em Extremes}, {\bf 13}, 399--437.

\bibitem[\protect\citename{J{\o}rgensen {\em et~al.\ }\relax,
  }2011]{Jorgensen2011a}
{\sc J{\o}rgensen, B., Dem{\'e}trio, C. G.~B., Kristensen, E., Banta, G.~T.,
  Petersen, H.~C., \& Delefosse, M.} 2011.
\newblock Bias-corrected {P}earson estimating functions for {T}aylor's power
  law applied to benthic macrofauna data.
\newblock {\em Statist. Probab. Lett.}, {\bf 81}, 749--758.

\bibitem[\protect\citename{Karlis \& Xekalaki, }2005]{Karlis2005}
{\sc Karlis, D., \& Xekalaki, E.} 2005.
\newblock Mixed {P}oisson distributions.
\newblock {\em Int. Statist Rev.}, {\bf 73}, 35--58.

\bibitem[\protect\citename{Kemp, }1997]{Kemp1997}
{\sc Kemp, A.~W.} 1997.
\newblock Characterizations of a discrete normal distribution.
\newblock {\em J. Statist.Plann. Inf.}, {\bf 63}, 223--229.

\bibitem[\protect\citename{Kemp \& Kemp, }1965]{Kemp1965}
{\sc Kemp, C.~D., \& Kemp, A.~W.} 1965.
\newblock Some properties of the `{H}ermite' distribution.
\newblock {\em Biometrika}, {\bf 52}, 381--394.

\bibitem[\protect\citename{Khatri, }1959]{Khatri1959}
{\sc Khatri, C.~G.} 1959.
\newblock On certain properties of power-series distributions.
\newblock {\em Biometrika}, {\bf 46}, 486--490.

\bibitem[\protect\citename{Kokonendji \& P{\'e}rez-Casany,
  }2012]{Kokonendji2012}
{\sc Kokonendji, C.~C., \& P{\'e}rez-Casany, M.} 2012.
\newblock A note on weighted count distributions.
\newblock {\em Journal of Statistical Theory and Applications}, {\bf 11},
  337--352.

\bibitem[\protect\citename{Kokonendji {\em et~al.\ }\relax,
  }2004]{Kokonendji2004}
{\sc Kokonendji, C.~C., Dossou-Gb{\'e}t{\'e}, S., \& Dem{\'e}trio, C. G.~B.}
  2004.
\newblock Some discrete exponential dispersion models: {P}oisson-{T}weedie and
  {H}inde-{D}emétrio classes.
\newblock {\em SORT: Statistics and Operations Research Transactions}, {\bf
  28}, 201--214.

\bibitem[\protect\citename{Kokonendji {\em et~al.\ }\relax,
  }2008]{Kokonendji2008}
{\sc Kokonendji, C.~C., Mizère, D., \& Balakrishnan, N.} 2008.
\newblock Connections of the {P}oisson weight function to overdispersion and
  underdispersion.
\newblock {\em J. Statist. Plann. Inf.}, {\bf 138}, 1287--1296.

\bibitem[\protect\citename{Mass{\'e} \& Theodorescu, }2005]{Masse2005}
{\sc Mass{\'e}, J.-C., \& Theodorescu, R.} 2005.
\newblock Neyman {T}ype {A} distribution revisited.
\newblock {\em Statistica Neerlandica}, {\bf 59}, 206--213.

\bibitem[\protect\citename{McKenzie, }1985]{McKenzie1985}
{\sc McKenzie, E.} 1985.
\newblock Some simple models for discrete variates time series.
\newblock {\em Water Resources Bulletin}, {\bf 21}, 645--650.

\bibitem[\protect\citename{Mora, }1990]{Mora1990}
{\sc Mora, M.} 1990.
\newblock La convergence des fonctions variance des familles exponentielles
  naturelles.
\newblock {\em Ann. Fac. Sci. Toulouse (5)}, {\bf 11}, 105--120.

\bibitem[\protect\citename{Pistone \& Wynn, }1999]{Pistone1999}
{\sc Pistone, G., \& Wynn, H.~P.} 1999.
\newblock Finitely generated cumulants.
\newblock {\em Statistica Sinica}, {\bf 9}, 1029--1052.

\bibitem[\protect\citename{Puig, }2003]{Puig2003}
{\sc Puig, P.} 2003.
\newblock Characterizing additively closed discrete models by a property of
  their maximum likelihood estimators, with an application to generalized
  {H}ermite distributions.
\newblock {\em J. Amer. Statist. Assoc.}, {\bf 98}, 687--692.

\bibitem[\protect\citename{Puig \& Barquinero, }2011]{Puig2011}
{\sc Puig, P., \& Barquinero, F.} 2011.
\newblock An application of compound Poisson modelling to biological dosimetry.
\newblock {\em Proc. Royal Society A}, {\bf 467}, 897--910.

\bibitem[\protect\citename{Puig \& Valero, }2006]{Puig2006}
{\sc Puig, P., \& Valero, J.} 2006.
\newblock Count data distributions: some characterizations with applications.
\newblock {\em J. Amer. Statist. Assoc.}, {\bf 101}, 332--340.

\bibitem[\protect\citename{Puig \& Valero, }2007]{Puig2007}
{\sc Puig, P., \& Valero, J.} 2007.
\newblock Characterization of count data distributions involving additivity and
  binomial subsampling.
\newblock {\em Bernoulli}, {\bf 13}, 544--555.

\bibitem[\protect\citename{Risti{\'c} {\em et~al.\ }\relax, }2009]{Ristic2009}
{\sc Risti{\'c}, M.~M., Bakouch, H.~S., \& Nasti{\'c}, A.~S.} 2009.
\newblock A new geometric first-order integer-valued autoregressive
  {(NGINAR(1))} process.
\newblock {\em Journal of Statistical Planning and Inference}, {\bf 139},
  2218--2226.

\bibitem[\protect\citename{Roy, }2003]{Roy2003}
{\sc Roy, D.} 2003.
\newblock The discrete normal distribution.
\newblock {\em Communications in Statistics---Theory and Methods}, {\bf 32},
  1871--1883.

\bibitem[\protect\citename{Rudin, }1976]{Rudin1976}
{\sc Rudin, W.} 1976.
\newblock {\em Principles of Mathematical Analysis}. third edn.
\newblock New York: McGraw-Hill.

\bibitem[\protect\citename{Sellers {\em et~al.\ }\relax, }2012]{Sellers2012}
{\sc Sellers, K.F., Borle, S., \& Shmueli, G.} 2012.
\newblock The {COM-P}oisson model for count data: a survey of methods and
  applications.
\newblock {\em Applied Stochastic Models in Business and Industry}, {\bf 28},
  104--116.

\bibitem[\protect\citename{Shmueli {\em et~al.\ }\relax, }2005]{Shmueli2005}
{\sc Shmueli, G., Minka, T.~P., Kadane, J.~P., Borle, S., \& Boatwright, P.}
  2005.
\newblock A useful distribution for fitting discrete data: revival of the
  {C}onway-{M}axwell-{P}oisson distribution.
\newblock {\em Applied Statistics}, {\bf 54}, 127--142.

\bibitem[\protect\citename{Steutel \& {van Harn}, }1979]{Steutel1979}
{\sc Steutel, F.~W., \& {van Harn}, K.} 1979.
\newblock Discrete analogues of self-decomposability and stability.
\newblock {\em Ann. Probab.}, {\bf 7}, 893--899.

\bibitem[\protect\citename{Tweedie, }1984]{Tweedie1984}
{\sc Tweedie, M. C.~K.} 1984.
\newblock An index which distinguishes between some important exponential
  families.
\newblock {\em Pages  579--604 of:} {\sc Ghosh, J.~K., \& Roy, J.} (eds), {\em
  Statistics: Applications and New Directions. Proceedings of the Indian
  Statistical Institute Golden Jubilee International Conference}.
\newblock Calcutta: Indian Statistical Institute.

\bibitem[\protect\citename{Wei{\ss}, }2008]{Weiss2008}
{\sc Wei{\ss}, C.~H.} 2008.
\newblock Thinning operations for modeling time series of counts---a survey.
\newblock {\em AStA Advances in Statistical Analysis}, {\bf 92}, 319--341.

\bibitem[\protect\citename{Willmot, }1987]{Willmot1987}
{\sc Willmot, G.~E.} 1987.
\newblock The {P}oisson-inverse {G}aussian distribution as an alternative to
  the negative binomial.
\newblock {\em Scand. Actuar. J.}, {\bf 1987}, 113--127.

\bibitem[\protect\citename{Wimmer \& Altmann, }1999]{Wimmer1999}
{\sc Wimmer, G., \& Altmann, G.} 1999.
\newblock {\em Thesaurus of Univariate Discrete Probability Distributions.}
\newblock Essen: STAMM Verlag.

\end{thebibliography}

\newpage
\end{document}